\documentclass[a4paper]{article}
\usepackage[english]{babel}
\usepackage[utf8x]{inputenc}
\usepackage[T1]{fontenc}
\usepackage{algorithmic}
\usepackage{algorithm}
\usepackage{euscript}
\usepackage{pgfplots,tikz,subfigure}
\usepackage{mathtools}

\usepackage[a4paper,top=3cm,bottom=2cm,left=3cm,right=3cm,marginparwidth=1.75cm]{geometry}

\usepackage{amsmath}
\usepackage{amsthm}
\usepackage{amssymb}
\usepackage{amsfonts}
\usepackage{graphicx}
\usepackage[colorinlistoftodos]{todonotes}
\usepackage[colorlinks=true, allcolors=blue]{hyperref}
\usepackage{cleveref}
\usepackage{cite}

\newcommand{\CASE}[1]{\STATE \textbf{case} #1\textbf{:} \begin{ALC@g}}
\newcommand{\ENDCASE}{\end{ALC@g}}

\newcommand{\DEFAULT}{\STATE \textbf{default:} \begin{ALC@g}}
\newcommand{\ENDDEFAULT}{\end{ALC@g}}
\newcommand{\DEFAULTLINE}[1]{\STATE \textbf{default:} }

\newcommand{\diag}{\mathop{\mathrm{diag}}}
\newcommand{\rank}{\mathop{\mathrm{rank}}}

\newcommand{\IR}{\mathbb{R}}
 \newcommand{\IC}{\mathbb{C}}

\newcommand{\htwo}{{\mathcal{H}_2}}
\newcommand{\hinf}{{\mathcal{H}_\infty}}
\newcommand{\pvec}{\mathsf{p}}
\newcommand{\Ffull}{\EuScript{F}}
\newcommand{\Hfull}{\EuScript{H}}
\newcommand{\Efull}{\EuScript{E}}
\newcommand{\Gfull}{\EuScript{G}}
\newcommand{\Hred}{\widehat{\Hfull}}
\newcommand{\Ered}{\widehat{\Efull}}
\newtheorem{theorem}{Theorem}
\newtheorem{proposition}{Proposition}
\newtheorem{corollary}{Corollary}
\newtheorem{example}{Example}
 \newtheorem{remark}{Remark}
\title{Sampling-free parametric model reduction \\
for structured systems}
\author{Christopher Beattie,  Serkan Gugercin, and Zoran Tomljanovi\'{c}}

\begin{document}
\maketitle
%
\begin{abstract}
We consider the reduction of parametric families of linear dynamical systems having an affine parameter dependence that differ from one another by
a low-rank variation in the state matrix. Usual approaches for parametric model reduction typically involve exploring the parameter space to isolate representative models on which to focus model reduction methodology, which are then combined in various ways in order to interpolate the response from these representative models.  The initial exploration of the parameter space can be a forbiddingly expensive task.   

A different approach is proposed here that does not require any parameter sampling or exploration of the parameter space. Instead, we represent the system response in terms of four subsystems that are \emph{nonparametric}.
One may apply any one of a number of standard (nonparametric) model reduction strategies to reduce the subsystems independently, and then conjoin these reduced models with the underlying parameterized representation to obtain an overall parameterized response. Our approach has elements in common with the parameter mapping approach of Baur et al. \cite{baur2014mapping}, but offers greater flexibility and potentially greater control over accuracy.  
In particular, a data-driven variation of our approach is described that exercizes this flexibility through the use of limited frequency-sampling of the underlying nonparametric models.  The parametric structure of our system representation allows for \emph{a priori} guarantees of system stability the resulting parametric reduced models, uniformly across all parameter values.  Incorporation of system theoretic error bounds allow us to determine appropriate approximation orders for the nonparametric systems sufficient to yield uniformly high accuracy with respect to parameter variation. 

We illustrate our approach on a class of structural damping optimization problems and on a benchmark model of  thermal conduction in a semiconductor chip. 
 The parametric structure of our reduced system representation lends itself very well to the development of optimization strategies making use of efficient cost function surrogates. We discuss this in some detail for damping parameter and location optimization for vibrating structures.

\end{abstract}
{{\textbf{Keywords}: parametric model reduction, sampling-free; damping optimization, structured systems\\}

\section{Introduction}
Consider a linear time invariant dynamical system, parameterized with a $k$-dimensional parameter vector 
$\pvec=\begin{bmatrix}p_1, p_2, \ldots, p_k\end{bmatrix}^T\in \Omega \subseteq \IR^{k}$ 
and represented in state-space form as
\begin{align}\label{Mainsystem}
 \begin{split}
 E  \dot x(t;\pvec) &= A(\pvec)x(t;\pvec) + B w(t),\\
y(t;\pvec) &= Cx(t;\pvec),
\end{split}
\end{align}
where  $E, A(\pvec) \in \mathbb{R}^{n\times n}$, $B\in \IR^{n\times m} $ and $C\in \IR^{\ell\times n}$ are constant (time-invariant) matrices. In \eqref{Mainsystem}, $ x(t;\pvec)\in \IR^{n} $, $ u(t)\in \IR^{m}$ and $y(t;\pvec)\in \IR^{\ell}$ 
denote,  the  state vector, inputs, and outputs, respectively.
We assume throughout that the matrix $E$ is invertible.

\subsection{Basic structure}
The structural feature of the system that we will exploit extensively presumes that the system matrix $A(\pvec)$ has the parametric form
\begin{equation}\label{Apstructure}
A(\pvec)=A_0-U\,\diag (p_1, p_2, \ldots, p_k)V^T=A_0-\sum_{i=1}^k p_i u_i v_i^T,
\end{equation}
where  $U=\begin{bmatrix}u_1, u_2, \ldots, u_k\end{bmatrix} \in \IR^{n \times k}$ and $V =\begin{bmatrix}v_1, v_2, \ldots, v_k\end{bmatrix}  \in \IR^{n \times k}$ are constant matrices with  $u_i, v_i\in \IR^{  k}$  for $i=1,\ldots,k$. Note that the parameterization in \eqref{Apstructure} is a special case of a general affine parametrization $A(\pvec) = A_0-\sum_{i} p_i A_i$ with the added rank constraint that the matrices $A_i$ have rank-$1$.  However, even for a general affine parametrization with $A_i \in \IR^{n\times n}$ individually having unrestricted rank but in aggregate having $\mathsf{rank}(\sum_{i} p_i A_i)=k \ll n$, the form of \eqref{Apstructure} may be assumed without loss of generality (allowing for the possibility that the parameters are replaced by functions of $\{p_i\}_i$). 
 The condition $k \ll n$ is a practical constraint leading to the prospect of computational efficiency, but there is no theoretical restriction on the size of $k$. 

Taking the Laplace transform of 
\eqref{Mainsystem}, the full-order transfer function of the parametrized system is obtained as 
\begin{equation}\label{TFfull}
\Hfull(s;\pvec)=C(s E -A(\pvec))^{-1}B = 
C\left(s E - \left( A_0-U\,\diag (p_1, p_2, \ldots, p_k)V^T\right)\right)^{-1}B.
\end{equation}
The goal of parametric model reduction, in this setting, is to
find a reduced parametric system 
\begin{align} \label{Redsystem}
 \begin{split}
 E_r  \dot x_r(t;\pvec) &= A_r(\pvec)x_r(t;\pvec) + B_r w(t),\\
y_r(t;\pvec) &= C_r x_r(t;\pvec),
 \end{split}
\end{align}
where $E_r, A_r(\pvec) \in \IR^{r\times r}$, $B_r\in \IR^{r\times m} $ and $C_r\in \IR^{\ell\times r}$ 
with $r \ll n$ such that the reduced transfer function
$$
\Hfull_r(s;\pvec)=C_r(s E_r -A_r(\pvec))^{-1}B_r
$$
approximates $\Hfull(s;\pvec)$ accurately for the parameter range of interest.

Several approaches to parametric model order reduction (\textsf{pMOR}) exist. One of the most common approaches involves state-space projection using globally defined bases: 
Choose a set of parameter samples $\pvec^1,\ldots, \pvec^{n_s}$. For every parameter sample $\pvec^i$, the full-order model $\Hfull(s;\pvec^k)$ becomes a non-parametric linear time-invariant system, for which 
a plethora of model reduction methods are available. 
Whatever choice is made, let $Z_r^i$ and $W_r^i$ denote the \emph{local} model reduction bases for the parameter sample $\pvec^i$, for each $i=1,\ldots,n_s$.  Then concatenate these local bases to form the \emph{global} model reduction bases $Z_r$ and $W_r$:   $Z_r=[Z_r^1,\ldots,Z_r^{n_s}]$ and $W_r=[W_r^1,\ldots, W_r^{n_s}]$. 
This concatenation step is usually followed by a rank-revealing $QR$ or truncated SVD computation 
to compute and condense orthogonal bases. 
The parametric reduced model quantities in \eqref{Redsystem} are obtained via a Petrov-Galerkin projection, i.e.,
\begin{align} \label{eq:redabce}
E_r=W_r^T E Z_r,\qquad   A_r(\pvec)=W_r^T A(\pvec) Z_r,\qquad
 B_r &=W_r^T B, \qquad \mbox{and} \quad   C_r=C Z_r.
\end{align}

Reviews of methods that consider such a reduction framework can be found, e.g. in \cite{morBauBBG11,BennerGW15,Antoulas01asurvey,Feng05,ANT05,BCOW17,QuarteroniMN16,BenMS05,BeattieGugercinSCL09,AntBG20}.
 These approaches are widely studied especially for structured systems with particular applications; see, e.g., \cite{BennerGW15,Antoulas01asurvey,BCOW17,QuarteroniMN16,YueMeer13SIAM,TomljBeattieGugercin18,BennerKuerTomljTruh15}.

The global basis approach has been successfully applied in many circumstances requiring parametric model reduction and in some cases it may be the only viable approach. Nonetheless it comes with some drawbacks, the main issue being the need to sample the parameter domain adequately in order to construct representative local bases. 
Except for special cases
\cite{morBauBBG11,grimm2018parametric}, how one chooses optimal parameter sampling points with respect to a joint global frequency-parameter error measure has not been known until recently. In \cite{morHunMS18},
Hund et al. tackles this joint-optimization problem by deriving optimality conditions and then
constructing model reduction bases that enforce those conditions.  
The most widely used approaches for 
global basis construction in  \textsf{pMOR}
are greedy or optimization-based sampling strategies; see \cite{BennerGW15} for a survey. However, especially in the case of high-dimensional parameter domains, this off-line 
stage could prove prohibitively expensive since it requires a large-number of full-order function evaluations. One may try to avoid these high-fidelity sampling techniques and pick parameter samples heuristically to make the off-line stage less costly. However, since the global bases directly depend on this initial sampling, which in turn influences the final accuracy of the parametric reduced model, if this stage is not done properly the reduced model could not be expected to provide a good approximation over a wide parameter range.

In this paper, we focus on systems having the special structure described in \eqref{Apstructure}. We develop a novel  parametric model order reduction approach that is sampling free (that is, there is no need for parameter sampling) yet it still offers uniformly high fidelity across the full parameter range. Significantly, the reduced model retains the parametric structure of the original full model.

\subsection{A motivating example: Damping optimization} 
\label{sec:motivexample}

Consider the  vibrational system  described by
\begin{align}\label{MDK}
 \begin{split}
M\ddot q(t)+D\dot q(t)+Kq(t)&=B_2 w(t),\\
y(t)&=C_2q(t),
 \end{split}
\end{align}
where $M$ and   $K$ are real, symmetric positive definite matrices of size $n\times n$, denoting the mass and stiffness matrices, respectively.  The state variables are described by the coordinate vector $q\in \IR^n$ representing structure displacements. The time dependent vector  $w(t)\in \IR^m$ is  the primary excitation and typically represents an input disturbance.  $B_2\in \IR^{n\times m}$ is the \emph{primary excitation matrix}, i.e., the input-to-state mapping. Similarly, $y(t)\in \IR^\ell$ is the performance output, representing a quantity of interest that is obtained from the state-vector $q$ via a mapping by the state-output matrix 
$C_2  \in
\IR^{\ell \times n}$. 

The damping matrix  $D \in\IR^{n \times n} $ is modeled as $$D=D_{int} + D_{ext},$$ where $D_{ext}$ represents  external damping and $D_{int}$ represents internal damping. The internal damping $D_{int}$ is usually taken to be a small multiple of the critical damping denoted by
$D_{crit}$  or a small multiple of proportional damping (see,  e.g.,  \cite{ BennerTomljTruh11, BennerKuerTomljTruh15}):
\begin{equation}\label{C_{int}}
  D_{int} = \alpha_c D_{crit},\quad \mbox{where} \quad D_{crit} =2 M^{1/2}\sqrt{M^{-1/2}KM^{-1/2}}M^{1/2}.
\end{equation}
Other possibilities for modeling internal damping can be found, e.g., in \cite{KuzmTomljTruh12}. 

We are mainly interested in in the external damping of the type $$D_{ext}=U_2\diag{(p_1, p_2, \ldots, p_k)}U_2^T,$$
where the non-negative entries $p_i$  for $i=1,\ldots,k$ represent the friction   coefficients  of the dampers, usually called gains or viscosities, and  the matrix $U_2$   encodes the 
damper positions and geometry; for more details, see, e.g., \cite{BennerKuerTomljTruh15, Blanchini12,TomljBeattieGugercin18,VES2011,BennerTomljTruh10,MullerSchiehlen85}.

In damping optimization problems, the   major goal is to determine \emph{best/optimal} external  damping
matrix $D_{ext}$ that will minimize the influence of the input $w$  on the output $y$. One can consider different optimality measures. In the input-output dynamical systems settings, the optimization criteria are usually based on system norms such as $\mathcal{H}_2$ or $\mathcal{H}_\infty$ system norm (see, e.g., \cite{Blanchini12, BennerKuerTomljTruh15}). Moreover, mixed performance measure that was introduced in \cite{NTT19}.  This specific choice of  optimization criteria strongly depends on the application at hand. \textsf{pMOR} methods we will develop in this paper will allow using different optimization criteria and therefore will enable different applications.  

By defining the state-vector as $x=[q^T~\dot q^T]^T$  we obtain a first-order state-space representation of the vibrational system:
\begin{align*}
E \dot x(t)&=A(\pvec) x(t)+Bw(t),   \\
y(t)&=Cx(t)\nonumber,
\end{align*}
where
\begin{align}  \label{msd1}
E&=\left[%
\begin{array}{cc}
  I & 0  \\
  0 & M \\
\end{array}%
\right], \quad  B=\left[%
\begin{array}{c}
  0 \\
 B_2 \\
\end{array}%
\right],\quad
C =\left[
         \begin{array}{cc}
           C_2 & 0 \\
         \end{array}
       \right],\\[1ex] \label{msd2}
 \mbox{and} \quad  A(\pvec)&=\underbrace{\left[%
\begin{array}{cc}
  0 & I \\
  -K & -D_{int} \\
\end{array}%
\right]}_{A_0}-\underbrace{\left[%
\begin{array}{c}
  0   \\
   U_2  \\
\end{array}%
\right]}_{U}\diag(p_1, p_2, \ldots, p_k) \underbrace{\left[\begin{array}{cc}
  0   &     U_2^T  \\
\end{array}\right]}_{V^T = U^T},
\end{align}
with $ \pvec=\begin{bmatrix}p_1, p_2, \ldots, p_k\end{bmatrix}^T$. Note that the model for damping optimization  in \eqref{msd1}-\eqref{msd2} has the parametric structure described in \eqref{Apstructure}.

The optimization of damper locations can be formulated effectively as optimization over a finite (but potentially large) number of configurations for the matrix $B_2$; this is a demanding combinatorial optimization problem and for each $B_2$-configuration, one must optimize over $\pvec$, the parameter vector.  \textsf{pMOR} approaches seeking to make this task cheaper have been considered previously: For optimization based on the $\mathcal{H}_2$ norm criterion, \cite{BennerKuerTomljTruh15} used a global basis approach, as described above, where local bases were obtained via the dominant pole algorithm \cite{RommesM08}. Using the same optimization criterion, \cite{TomljBeattieGugercin18} applied a global basis approach where local bases were obtained 
via the Iterative Rational Krylov Algorithm (\textsf{IRKA}) \cite{GugAB08}, an $\mathcal{H}_2$-optimal model reduction approach. Even though both approaches show great promise,  success in each case depends on the initial parameter sampling used to construct the global basis, an issue that is faced in most \textsf{pMOR} cases. For the damping optimization
problem, an efficient heuristic that can guide a parameter sampling strategy is not available, and the natural alternative, a preliminary offline greedy sampling stage, can be computationally very demanding and potentially negate the gains one would anticipate from model reduction.

In subsequent sections, we will propose two frameworks 
that will remove the need for parametric sampling in problems structured
 as in \eqref{Apstructure}, including in particular the damping optimization problem discussed above. Thus, the new approach will allow efficient optimization of damping parameters encoded in the  vector $\pvec$.

\section{PMOR based on subsystem model reduction} \label{sec:method1}
In this section, we introduce our first sampling-free reduction method for the structured problem \eqref{TFfull}. We provide error bounds and also discuss the uniform stability of the reduced model.
\subsection{Reformulation of the parametric transfer function}
The crucial observation and the starting point behind our framework is that we can rewrite  the structured transfer function \eqref{TFfull} in a form that separates the $s$ and $\pvec$ dependency, by making use of the  Sherman-Morrison-Woodbury formula \cite{GVL89}. We summarize this result in the following proposition. 
\begin{proposition} \label{swmonH}
Consider the structured transfer function 
\begin{equation} \label{TFfull2}
\Hfull(s;\pvec)= C\left(s E - \left( A_0-U\,\diag (p_1, p_2, \ldots, p_k)V^T\right)\right)^{-1}B,
\end{equation}
where $p_i \in \IR_{+}$, for $i=1,2,\ldots,k$. Then, 
\begin{align}\label{TFSMWstile}
\Hfull(s;\pvec)&=\Hfull_1(s)- \Hfull_2(s) D(\pvec)  \left[I+ D(\pvec) \Hfull_3(s) D(\pvec)\right]^{-1} D(\pvec) \Hfull_4(s),
\end{align}
where the diagonal matrix 
\begin{equation} \label{Dp}
D(\pvec)= \diag (\sqrt p_1,\sqrt p_2, \ldots,\sqrt p_k)
\end{equation}
encodes the parameters, and 
$\Hfull_1(s)$, $\Hfull_2(s)$, $\Hfull_3(s)$, and $\Hfull_4(s)$
are non-parametric transfer functions given by
\begin{align}\label{H1-H4}
 \begin{split}
\Hfull_1(s)&=C  (s E - A_0)^{-1}B,\phantom{~\mbox{and}}   \qquad\,\,\Hfull_2(s)  = C  (s E - A_0)^{-1}U,\\[2mm]
\Hfull_3(s)&=V^T  (s E - A_0)^{-1}U,~\mbox{and}  \qquad\Hfull_4(s)  = V^T (s E - A_0)^{-1}B.
 \end{split}
\end{align}
\end{proposition}

\begin{proof}
Let $T \in \IC^{n\times n}$ be invertible. Also, let
$X \in \IC^{n\times k}$ and $Y \in \IC^{n\times k}$ be such that  
$I_k + Y^T T^{-1}X$ is invertible. Then,
Sherman-Morrison-Woodbury formula states that
$$
(T + X Y^T )^{-1} = T^{-1} - T^{-1} X (I_k + Y^T T^{-1}X)^{-1}Y^T T^{-1}. 
$$
Recall that  
$ \Hfull(s;\pvec)= C\left(s E - \left( A_0-U\,\diag (p_1, p_2, \ldots, p_k)V^T\right)\right)^{-1}B$. The result, then, follows from the Sherman-Morrison-Woodbury formula by defining $T = s E - A_0$, $X = U D(p)$, and $Y = VD(p)$.
\end{proof}

\begin{remark} 
Notice that if we define the extended parameterized transfer function
$$
\widetilde{\Ffull}(s;\pvec)=
\left[\begin{array}{cc}
\Hfull_1(s) &  \Hfull_2(s) D(\pvec) \\[2mm]
D(\pvec) \Hfull_4(s) & I+ D(\pvec) \Hfull_3(s) D(\pvec)
\end{array}\right],
$$
then \eqref{TFSMWstile} is the Schur complement of 
$\widetilde{\Ffull}(s;\pvec)$ with respect to the (2,2) block:
$$
\Hfull(s;\pvec)=\left[\widetilde{\Ffull}(s;\pvec)/(I+ D(\pvec) \Hfull_3(s) D(\pvec))\right].
$$
\end{remark}

\begin{remark}
Motivated by the damping optimization problem described in Section \ref{sec:motivexample}, Proposition \ref{swmonH}  assumes that the parameter vector $\pvec$ has positive entries, which lead to the form 
\eqref{TFSMWstile} where the diagonal matrix $D(p)$ appear in a balanced symmetric way throughout the second term. However, the positive parameter range assumption is not necessary and  the general case could be easily handled in a similar way. Define 
$$\widetilde{D}(\pvec) = {\rm diag}(p_1,p_2,\ldots,p_k).$$ Then,
\begin{align}\label{TFSMWstilenegativep}
\Hfull(s;\pvec)&=\Hfull_1(s)- \Hfull_2(s) \widetilde{D}(\pvec)  \left[I+\Hfull_3(s) \widetilde{D}(\pvec)\right]^{-1} \Hfull_4(s),
\end{align}
where $\Hfull_1(s)$, $\Hfull_2(s)$, $\Hfull_3(s)$, and $\Hfull_4(s)$ are as defined 
in \eqref{H1-H4}. In the rest of the paper, we will use the formulation in
Proposition \ref{sec:motivexample}, but all the results to follow can be easily generalized using the form \eqref{TFSMWstilenegativep}.
\end{remark}

\subsection{Subsystem model reduction}
\Cref{swmonH} displays a decomposition of the full-order transfer function $\Hfull(s;\pvec)$ in terms of four \emph{non-parametric} transfer functions with the parameter dependency entering as an interconnection coupling the four systems. Since $\Hfull_i(s)$, for $i=1,2,3,4$, are non-parametric, they may be reduced without any need for sampling via well-established model reductions techniques such as balanced truncation (\textsf{BT}) \cite{mullis1976synthesis,moore1981principal}, Hankel norm approximation (\textsf{HNA}) \cite{glover1984all}, or iterative rational Krylov algorithm \textsf{IRKA} \cite{GugAB08}; see \cite{ANT05,BCOW17,Antoulas01asurvey}
for further choices.  Moreover, each system $\Hfull_i(s)$, may be reduced independently of the others, potentially using different reduction orders and even different reduction methodologies.

Let the reduced model for $\Hfull_i(s)$ be denoted by $\Hred_i(s)$, for $i=1,\ldots,4$. 
The resulting parametric reduced model for 
$\Hfull(s;\pvec)$ is given by 
  \begin{equation}\label{TFSMWreduced}
   \Hfull(s;\pvec) \approx \Hred(s;\pvec) = \Hred_1(s)- \Hred_2(s) D(\pvec)  (I+ D(\pvec) \Hred_3(s) D(\pvec))^{-1} D(\pvec) \Hred_4(s).
  \end{equation}
The online evaluation (or simulation) of   $\Hred(s;\pvec)$ for a given parameter value is trivial as it only involves reduced quantities and evaluation of the matrix $D(\pvec)$. Therefore, we have constructed a parametric, easy-to-evaluate reduced model without any need for parameter sampling. \Cref{Method1} below gives a  sketch of this process.

\begin{algorithm}[htp] 
 \caption{Parametric reduced order model based on reduction of subsystems}   
\label{Method1}                                            
\begin{algorithmic} [1]                                        

\STATE  \label{1st} \textsf{Off-line Stage:} Calculate the four non-parametric reduced systems
          \begin{align*}
      \Hfull_1(s) \rightarrow \Hred_1(s),\hspace{2ex}
    \Hfull_2(s) \rightarrow \Hred_2(s), \hspace{2ex}
        \Hfull_3(s)  \rightarrow \Hred_3(s), \hspace{2ex}
    \Hfull_4(s)   \rightarrow \Hred_4(s),
     \end{align*}
     (Reductions can be performed via a variety of nonparametric reduction techniques)  
\STATE \label{2nd}\textsf{On-line Stage:} For any given parameter $\pvec=\begin{bmatrix}p_1, p_2, \ldots, p_k\end{bmatrix}^T$, obtain the parametric model by
$$   \Hfull(s;\pvec) \approx \Hred(s;\pvec) = \Hred_1(s)- \Hred_2(s) D(\pvec)  (I+ D(\pvec) \Hred_3(s) D(\pvec))^{-1} D(\pvec) \Hred_4(s).
 $$
\end{algorithmic}
\end{algorithm}

Some remarks are in order regarding 
Step \ref{1st} in \Cref{Method1}.
The model $\Hfull_1(s)$ has the same input-output dimension as $\Hfull(s;\pvec)$. The model $\Hfull_2(s)$ has the same number of outputs ($\ell$) as $\Hfull(s;\pvec)$ and $k$ inputs. Similarly,  the model $\Hfull_4(s)$ has the same number of inputs ($m$) as $\Hfull(s;\pvec)$, and also $k$ outputs. Provided that the input/output dimension is modest, reducing $\Hfull_1(s)$, $\Hfull_2(s)$, and 
$\Hfull_4(s)$ will not be expected to be strongly influenced  by the size of $k$ since in most cases the smaller of the input/output dimensions determines the difficulty in reducing a dynamical systems. On the other hand, the model $\Hfull_3(s)$ will have $k$-inputs and $k$-outputs. Therefore, if $k$ is significantly  larger than $\ell$ or $m$, this is likely to be the most difficult model to reduce with high fidelity.  Therefore, although
the framework and theoretical analysis we develop here apply to a system with transfer function
$$
\Hfull(s;\pvec) = C(s E - (A_0- p A_1))^{-1} B
$$
\emph{even when $A_1$ has full rank}, computational difficulties might arise since
$\Hfull_3(s)$ will then be an $n$-dimension dynamical system with $n$-inputs and $n$-outputs. Such models will not generally be amenable to model reduction in most scenarios. 

Next we provide an error bound for the parametric model reduction due to \Cref{Method1}.
\begin{theorem} \label{upper bound}
Let the full order transfer function $\Hfull(s; \pvec)$, and the corresponding subsystems $\Hfull_i(s)$, for $i=1,\ldots,4$  be given as in \eqref{TFSMWstile} and \eqref{H1-H4}. Assume that
the nonparametric reduced models $\Hfull_i(s)$ are reduced so that
\begin{equation}
\|\Hfull_i - \Hred_i \| \leq \epsilon_i,~~~\mbox{for}~~~i=1,\ldots,4,
\end{equation}
and that the corresponding parametric reduced model $\Hred(s;\pvec)$  is constructed as in   \eqref{TFSMWreduced}. Then,
\begin{align}
\label{upperbound}
 \|\Hfull(\cdot; \pvec) )-\Hred(\cdot;\pvec))\|\leq   \epsilon_1&   + f_1(\pvec,\Hred_3,\Hred_4) \epsilon_2  +
 f_1(\pvec,\Hred_3,\Hred_4)  f_2(\pvec,\Hfull_2,\Hfull_3)\epsilon_3 + 
 f_2(\pvec,\Hfull_2,\Hfull_3)\epsilon_4
\end{align}
where 
 \begin{align}\label{DefFunf1}
f_1(\pvec,\Gfull_1,\Gfull_2)&=\|D(\pvec)\big (I+D(\pvec)\Gfull_1(\cdot)D(\pvec)\big )^{-1} D(\pvec)\Gfull_2(\cdot)\|,~~\mbox{and}\\
f_2(\pvec,\Gfull_1,\Gfull_2   )&=\|\Gfull_1(\cdot) D(\pvec) \big ( I+D(\pvec)\Gfull_2(\cdot) D(\pvec) \big )^{-1}D(\pvec) \|.\label{DefFunf2}
\end{align}
\end{theorem}
\begin{proof}
First, by using the formulae \eqref{TFSMWstile} and    \eqref{TFSMWreduced}, we obtain 
\begin{align}\label{auxiliary1}
 \Hfull(s;\pvec)-\Hred(s;\pvec)=&  \Hfull_1(s)-\Hred_1(s) +\Hred_2(s)\Ered_1-  \Hfull_2(s) \Efull_1,
\end{align}
where  $\Efull_1=D(\pvec)  (I+ D(\pvec) \Hfull_3(s) D(\pvec))^{-1} D(\pvec) \Hfull_4(s)$ and $\Ered_1= D(\pvec)  (I+ D(\pvec) \Hred_3(s) D(\pvec))^{-1} D(\pvec) \Hred_4(s)$.
The last two terms can be manipulated as 
\begin{align}\nonumber
 \Hred_2(s) \Ered_1- \Hfull_2(s) \Efull_1 &= [\Hred_2(s)-\Hfull_2(s)]\Ered_1+ \Hfull_2(s) (\Ered_1-\Efull_1)\\ \nonumber
 &= [\Hred_2(s)-\Hfull_2(s)]\Ered_1+ \Hfull_2(s) [\Ered_2\Hred_4(s)-\Efull_2\Hfull_4(s)], 
\end{align}
where $ \Efull_2 =D(\pvec)  (I+ D(\pvec) \Hfull_3(s) D(\pvec))^{-1} D(\pvec)$ and $\Ered_2=D(\pvec)  (I+ D(\pvec) \Hred_3(s) D(\pvec))^{-1} D(\pvec) $. This last expression can be rewritten as
\begin{align}
 \Hred_2(s) \Ered_1-  \Hfull_2(s) \Efull_1 &= [ \Hred_2(s)-\Hfull_2(s)]\Ered_1+ \Hfull_2(s) [(\Ered_2-\Efull_2)\Hred_4(s)+\Efull_2(\Hred_4(s)-\Hfull_4(s) )]\label{auxiliary2}.
\end{align}
 Next consider the term $ \Efull_2-\Ered_2$, which can be expressed as
{\small
\begin{align}\label{widehatE2-E2}
\Ered_2-\Efull_2=D(\pvec) (I+ D(\pvec) \Hred_3(s) D(\pvec))^{-1} D(\pvec)[\Hfull_3(s)-\Hred_3(s)]D(\pvec)  (I+ D(\pvec)\Hfull_3(s) D(\pvec))^{-1} D(\pvec).
\end{align}
}

\noindent
Substituting \eqref{widehatE2-E2} into \eqref{auxiliary2}, which is then  substituted into \eqref{auxiliary1}, yields
{\small
\begin{align*}
 \Hfull(\cdot;\pvec)-\Hred(\cdot;\pvec)=  [\Hfull_1(s) - \Hred_1(s)]&   + [\Hfull_2(s) - \Hred_2(s) ]
D(\pvec)\big (I+D(\pvec)\Hred_3(s)D(\pvec)\big )^{-1} D(\pvec)\Hred_4(s)\\
+\Hfull_2(s)& D(\pvec)\big ( I+D(\pvec)\Hfull_3(s) D(\pvec) \big )^{-1}D(\pvec) [\Hred_4(s) - \Hfull_4 (s) ]+\\
+&\Hfull_2(s) D(\pvec)\big ( I+D(\pvec)\Hfull_3(s) D(\pvec) \big )^{-1}D(\pvec)[\Hfull_3(s) - \Hred_3(s)  ] \\
&\qquad\qquad \cdot D(\pvec)\big ( I+D(\pvec)\Hred_3(s) D(\pvec) \big )^{-1}D(\pvec)\Hred_4(s).
\end{align*}
}
The upper bound follows by taking norms on both sides.
\end{proof}

\subsection{Uniform stability of the parametric reduced model} 
\label{sec:stability}
In many applications, the full model $\Hfull(s,\pvec)$ in \eqref{TFfull2}
 is asymptotically stable for every $\pvec \in \Omega$, so that eigenvalues of the matrix pencil
$\lambda E - A(\pvec)$ (poles of the transfer function $\Hfull(s,\pvec)$)
have negative real parts for every $\pvec \in \Omega$. 
The damping optimization problem we considered in Section \ref{sec:motivexample} is one such example where the underlying physical phenomenon is asymptotically stable for every parameter in the parameter domain of interest.  Therefore, it is reasonable to expect the same of the parametric reduced model. We will call this \emph{uniform asymptotic stability} of the reduced model. Unfortunately, except for some special cases (e.g., $E=E^T$ is positive define, $A(\pvec) = A^T(\pvec)$ is negative definite, and one chooses $W_r = Z_r$ in \eqref{eq:redabce}), this property is difficult to  enforce. We refer the reader to \cite[\S 5.4]{BennerGW15} for a brief discussion on this issue.
We show here that our proposed framework will guarantee uniform stability for a significantly broader class of problems. 

Recall the damping optimization model with the space-state quantities defined in
\eqref{msd1}-\eqref{msd2}. Due to the internal damping term, $D_{int}$,  eigenvalues of the matrix pencil $s E - A_0$ have negative real parts; and thus all the subsystems, i.e., $\Hfull_1(s)$, $\Hfull_2(s)$, $\Hfull_3(s)$, and $\Hfull_4(s)$, are asymptotically stable. Then, in reducing these nonparametric subsystems, we can enforce stability preservation by employing, for example, \textsf{BT}~\cite{moore1981principal,mullis1976synthesis}, \textsf{HNA}~\cite{glover1984all}, or \textsf{IRKA} with stability enforcement \cite{gugercin2008isk,GugAB08,beattie2009trm}. Therefore, starting point of our uniform stability result is that all the subsystems are asymptotically stable, as is in the damping  optimization problem. 

To further motivate the setting of our uniform stability result, next we inspect the term $\Hfull_3(s)$ more closely for the damping optimization problem. It directly follows from \eqref{msd1} and \eqref{msd2} that the transfer function $\Hfull_3(s)$ is given by $\Hfull_3(s) = s U_2^T(s^2 M + s D + K)^{-1} U_2$, which corresponds to a positive real
(passive) dynamical system. In other words, $\textsf{Re}(\Hfull_3(s)) \geq 0$ for all $s$ with 
$\textsf{Re}(s) \geq 0$. Therefore,  $I+D(\pvec)\Hfull_3(s)D(\pvec)$ is strictly positive real, i.e., $\textsf{Re}(\Hfull_3(s)) > 0$ for $\textsf{Re}(s) \geq 0$. Using positive real balanced truncation \cite{desai1984transformation,ober1991balanced} or interpolatory port-Hamiltonian model reduction \cite{GugPBS12,EggKLMM18,polyuga2011structure}, one can reduce $\Hfull_3(s)$ in a way to retain positive realness. Motivated by these structures and observations appearing in the  damping optimization problem, we know state the uniform stability result.
\begin{theorem}  \label{thm:stable}
Consider the full parametric model with its transfer function
$\Hfull(s)$ written in terms of the subsystems as in \eqref{TFfull2} where the subsystems $\Hfull_1(s)$,
$\Hfull_2(s)$, and $\Hfull_4(s)$ are asymptotically stable. Further assume that $\Hfull_3(s)$
is positive real. Construct the reduced subsystems
$\Hred_1(s)$, $\Hred_2(s)$, and $\Hred_4(s)$ such that they retain asymptotic stability, and   $\Hred_3(s)$ such that it retains positive-realness.
Then, the reduced parametric model $\Hred(s)$  in \eqref{TFSMWreduced} is uniformly asymptotically stable for 
every $\pvec \in \Omega$.
\end{theorem}
\begin{proof}
It follows from the structure of the reduced transfer function
$\Hred(s,\pvec)$ in \eqref{TFSMWreduced} that  its poles are composed of the poles of $\Hred_1(s)$, $\Hred_2(s)$, and $\Hred_4(s)$, and the zeroes of $I+D(\pvec)\Hred_3(s) D(\pvec)$. Note that the poles of 
$\Hred_1(s)$, $\Hred_2(s)$, and $\Hred_4(s)$ are independent of $\pvec$ and all have negative real parts since these subsystems are asymptotically stable as they have been obtained via a model reduction algorithm to enforce this property. Therefore, these poles cannot contribute to potential instability and what is left to verify is that the zeroes of $I+D(\pvec)\Hred_3(s) D(\pvec)$ have negative real parts for every $p\in \Omega$. Recall that $\Hred_3(s)$ was
constructed with a positive-realness preserving model reduction technique. Therefore, $I+D(\pvec)\Hred_3(s) D(\pvec)$ is \emph{strictly} positive real for every $\pvec\in \Omega$ and cannot lose rank in the right-half plane, i.e., it does not have a zero in the right-half plane. Indeed,  $I+D(\pvec)\Hred_3(s) D(\pvec)$ is a minimum-phase 
system, meaning all of its poles and zeros have negative real parts.
Then, put together with the poles of $\Hred_1(s)$, $\Hred_2(s)$, and $\Hred_4(s)$, the reduced transfer function $\Hred(s,\pvec)$ is uniformly asymptotically stable. Note that there could be further pole-zero cancellations in the construction of $\Hred(s,\pvec)$. However, this will not change the conclusion since all the poles (before any potential pole-zero cancellation) already have negative real parts.
\end{proof}

\begin{remark}
One-approach to guarantee uniform asymptotic stability in the general case was proposed by Baur and Benner \cite{BauB09} where 
$\Hfull(s,\pvec)$ is sampled  at some parameter values $\pvec^i$ for $i=1,2,\ldots,n_s$ and reduced using stability preserving reduction such as \textsf{BT}. Then, the final parametric reduced system is obtained by connecting  these local reduced models via interpolation in the $\pvec$-domain. Since the interpolation in the $\pvec$-domain does not affect the poles, the resulting reduced parametric system is uniformly asymptotically stable. However, this comes at the cost of  poles being \emph{fixed}, i.e., the poles  do not vary with the parameters since the $\pvec$-dependency is completely in the $B$- or $C$-matrix. The situation  is different in our formulation where the poles still vary with $\pvec$, as desired, yet remain uniformly asymptotically stable. Moreover, the main motivation behind our approach here is to avoid the need for parameter sampling. 
\end{remark}
\begin{remark}
Theorem \ref{thm:stable} implies that in the case of the damping optimization problem (and others with similar structure), we can construct a sampling-free parametric reduced model that is uniformly asymptotically stable with an error bound.
\end{remark}

\subsubsection{Uniform asymptotic stability in the general case}
Theorem \ref{thm:stable} provided uniform asymptotic stability for the case when $\Hfull_3(s)$ is a positive real transfer function. What can we state about stability in the general case without this structure? 

We will continue to assume that all the subsystems are asymptotically stable and we apply an appropriate model reduction technique so that all the reduced subsystems are asymptotically stable as well. Following the proof of Theorem \ref{thm:stable}, we  only need to check the zeroes of $I+D(\pvec)\Hred_3(s) D(\pvec)$ and argue that these zeroes have negative real parts. 

To see the arguments more easily, let us consider the case that the parameter $\pvec$ is a scalar, $D(\pvec) =\sqrt{\pvec}$,  $U$ and $V$ are column vectors, and thus $\Hfull_3(s)$ is a single-input/single-output dynamical system. In this case, if $z_0$ is
zero of $1+D(\pvec)\Hred_3(s) D(\pvec)$, then $\Hred_3(z_0) = -1/\pvec.$
Let $\pvec_{\max}$ denote $\pvec_{\max} = \sup_{\pvec \in \Omega} \mid \pvec \mid$.
A sufficient condition for such a $z_0$ not to exist is that $\| \Hred_3 \|_\hinf < \frac{1}{\pvec_{\max}}$. Assuming that this condition holds for $\Hfull_3(s)$, which will guarantee that the full model is uniformly asymptotically stable, one can apply bounded real balancing \cite{opdenacker1988contraction,ober1991balanced}  in constructing $\Hred_3(s)$ so that it has the same $\hinf$-norm bound. The general case, i.e., when $\pvec$ is a vector, follows similarly and we list this result as a corollary.

\begin{corollary}
Consider the full parametric model with its transfer function
$\Hfull(s)$ written in terms of the subsystems as in \eqref{TFfull2} where all the subsystems $\Hfull_i(s)$ for $i=1,\ldots,4$ 
are asymptotically stable. 
Let $\pvec_{\max}$ denote $\pvec_{\max} = \sup_{\pvec \in \Omega} \| \pvec \|_\infty$ and  assume that $\Hfull_3(s)$
satisfies $\| \Hfull_3 \|_\hinf < \frac{1}{\pvec_{\max}}$. Construct the reduced subsystems
$\Hred_1(s)$, $\Hred_2(s)$, and $\Hred_4(s)$ such that they retain asymptotic stability, and   $\Hred_3(s)$ such that it retains 
the property $\| \Hred_3 \|_\hinf < \frac{1}{\pvec_{\max}}$.
Then, the reduced parametric model $\Hred(s)$  in \eqref{TFSMWreduced} is uniformly asymptotically stable for 
every $\pvec \in \Omega$.
\end{corollary}

\subsection{The parameter mapping approach of Baur et al.\cite{baur2014mapping}}
 The parameterization given in \eqref{Apstructure} appears as a (relatively) low rank change from the base dynamic matrix, $A_0$, and it is this structural feature that we have exploited.  Another strategy that exploits this parametric structure has been proposed in \cite{baur2014mapping}.  
 There, one augments and modifies the system by introducing a set of $k$ additional synthetic inputs, $\omega(t)$, and outputs, $\eta(t)$, in such a way that the internal system parameterization is mapped  to a feedthrough term. The original system response is recovered by constraining the synthetic inputs so as to null the synthetic outputs.  The modified system can be reduced independently of the parameterization and the final parameterized reduced model is recovered by imposing an analogous constraint on the synthetic inputs; they are chosen to null the (reduced) synthetic outputs.  To illustrate, define
 \begin{align} \label{baur_pvalMap}
 \begin{split}
 E  \dot x(t) &= A_0 x  + [B\ \  UD(\pvec)]
  \begin{bmatrix} w(t)\cr \omega(t)\end{bmatrix},\\
\begin{bmatrix}\hat{y}(t)\cr \eta(t) \end{bmatrix} &= \begin{bmatrix} C \cr D(\pvec)V^T \end{bmatrix} x(t) + \begin{bmatrix} 0 &  \cr  & I \end{bmatrix}
  \begin{bmatrix} w(t)\cr \omega(t)\end{bmatrix}.
\end{split}
\end{align}
Evidently, this system has $m+k$ inputs, $\ell+k$ outputs, and the parameterization now acts on a $k$ dimensional subpace that is common to both input and output spaces. 
Notice in particular that the parameterization no longer acts directly on the state vector.    What relation does the response of \eqref{baur_pvalMap} have with that of the original system \eqref{Apstructure} ? If $\omega(t)$ is chosen so that $\eta(t)=0$
(for example, if $\omega(t)$ is assigned by state feedback as $\omega(t)=-D(\pvec)V^T x(t)$), then the remaining output $\hat{y}(t)$ matches the output of the parameterized system described in  \eqref{Apstructure}: $\hat{y}(t)=y(t;\pvec)$. Indeed, with this added  constraint imposed on the synthetic inputs, the transfer function for the resulting system is identical to what has been  defined by 
\eqref{TFfull2}. 

The dynamical system described in \eqref{baur_pvalMap} may be reduced using any strategy appropriate for linear time-invariant MIMO (multiple input/multiple output) systems. Since the parameterization has been mapped to the synthetic input/output spaces and is now external to system dynamics, model reduction strategies can be pursued without the need of any parameter sampling. The approach described in \cite{baur2014mapping} proposes a projective reduced model derived, say, as in \eqref{Redsystem}-\eqref{eq:redabce} using projection bases defined by $Z_r$ and $W_r$:
\begin{align} \label{baur_pvalMapRed}
 \begin{split}
 E_r  \dot x_r(t) &= A_{0r} x_r  + [B_r\ \  W_r^TUD(\pvec)]
  \begin{bmatrix} w(t)\cr \hat{\omega}(t)\end{bmatrix},\\
\begin{bmatrix}\hat{y}_r(t)\cr \eta_r(t) \end{bmatrix} &= \begin{bmatrix} C_r \cr D(\pvec)V^TZ_r \end{bmatrix} x_r(t) + \begin{bmatrix} 0 &  \cr  & I \end{bmatrix}
  \begin{bmatrix} w(t)\cr \hat{\omega}(t)\end{bmatrix}.
\end{split}
\end{align}
Following \cite{baur2014mapping}, we  set up a state feedback constraint to null the reduced synthetic outputs, similar to what has gone before. If $\hat{\omega}(t)$ is assigned via reduced state feedback as $\hat{\omega}(t)=-D(\pvec)V^TZ_r x_r(t)$, then $\eta_r(t)=0$, and the reduced output $\hat{y}_r(t)$ will define the output of a reduced parametric system \eqref{Apstructure}: 
$y_r(t;\pvec)=\hat{y}_r(t).$ The state-space representation of the resulting reduced model is given my 
\begin{align} 
 W_r^T E Z_r  \dot x_r(t) &= W_r^T A_{0} Z_r x_r  -
  W_r^T U D^2(\pvec)Z_r + W_r^T B  w(r)\\
\hat{y}_r(t) &=  C Z_r x_r(t) 
\end{align}
To make a clearer connection to our proposed framework, we re-write the transfer function of this 
reduced dynamical system equivalent  as in \eqref{TFSMWreduced} as
$$
\Hred(s;\pvec) = \Hred_1(s)- \Hred_2(s) D(\pvec)  (I+ D(\pvec) \Hred_3(s) D(\pvec))^{-1} D(\pvec)\Hred_4(s).
$$
In this case each system, $\Hfull_i(s)$, is reduced with the projection bases, $Z_r$ and $W_r$, for $i=1,2,3,4$; that is,
the \emph{same projection bases} are used for all four systems.

Although the approach we take here is evidently closely aligned with the approach of \cite{baur2014mapping}, an important distinction is that we are able here to reduce the individual subsystems, $\Hfull_i(s)$, independently of one another.  This allows us to better control the fidelity of the final model, and as we described in Section \ref{sec:stability}, we are able to guarantee asymptotic stability of $\Hred(s;\pvec)$ uniformly in $\pvec$, so long as each reduced model $\Hred_i(s)$, $i=1,2,3,4$ is asymptotically stable and the single reduced $\Hred_3(s)$ is also positive real. 
Naturally, these assertions may also be made with the approach of \cite{baur2014mapping} or its recent formulation for second-order systems \cite{van2019parametric}, but it can be substantially more difficult to guarantee these properties using a single choice of projecting bases, $Z_r$ and $W_r$. 
We are able to exploit the flexibility of reducing the four subsystems independently of one another and do not suffer under these constraints. 

\subsection{Numerical Examples}

We will illustrate the performance of Algorithm \ref{Method1} on two numerical examples. In both examples, subsystem reduction was performed by \textsf{BT}.
Since the  subsystems have the same $E$-term and $A$-term, 
only one (Schur) decomposition in the offline stage,
Step \ref{1st}, of Algorithm \ref{Method1} is needed.
Therefore 
Step \ref{1st} is significantly cheaper than applying \textsf{BT} to four different systems.

\begin{example} { \em  
We consider the parametric version of the Penzl model \cite{Penzl99,IonitaAntoulas14}. The full model
transfer function
$
\Hfull(s,\pvec) = C(sE-A(\pvec)^{-1}B
$
is defined by the quantities $E=I$;
\begin{align*}
A&=\diag(A(p_1), A(p_2), A(p_3),-1,-2,\ldots,-M), \\
 &\mbox{where} \quad A(p_i)=\begin{bmatrix}
-1 &p_i\\
-p_i& -1\\
\end{bmatrix},\quad \mbox{for}\quad i=1,\ldots,3;\\
 C = [c_1~c_2~\ldots~c_M]& \in   \IR^{1\times (M+6)}\quad \mbox{where}\quad c_i= \left\{ \begin{array}{ll} 10 , \quad & i=1,\ldots,6 \\
   1  , \quad & i=7,\ldots,M;
                 \end{array}\right.
                 \quad 
                 \mbox{and} \quad B=C^T.
                \end{align*}
The parameters $p_1,p_2, p_3$ represent magnitude of the imaginary parts of the two eigenvalues of the diagonal block $A (p_i)$, for $i=1,2,3$, respectively and they control the location of
the peaks in the frequency response.

This linear time invariant system can be equivalently represented in the structured form  \eqref{Apstructure} with
\begin{equation*} 
A(\pvec)=A_0-U\,\diag (p_1, p_1, p_2,p_2, p_3,p_3)V^T, 
\end{equation*}
where $A_0 =\diag(I_6,-1,-2,\ldots,-M)$;  $v_i=e_i$ for $i=1,\ldots,6$ where $e_i$ is the $i$th canonical vector;  and similarly   
\begin{align*}
  & \quad u_i= \left\{ \begin{array}{ll} -e_{i+1} , \quad & i=1,3,5, \\
  e_{i-1}    , \quad & i= 2,4,6.
                 \end{array}\right.
\end{align*}
We chose $M=100$, which implies $n=106$. This modest dimension is not necessary as \textsf{BT} can be applied in much higher dimensions. It is chosen just to illustrate the theoretical considerations.  
Based on the Hankel singular values of $\Hfull_i(s)$, for $i=1,\ldots,4$,
in Step \ref{1st} of Algorithm \ref{Method1},  we obtain reduced subsystems $\Hred_i(s)$ with dimensions  $r_1=10$, $r_2 = 1$, $r_3=6$, and $r_4=1$, respectively. This is an advantage of the proposed framework: reduction of subsystems and their reduced orders are independent of each other. 

We compare $\Hfull(s,\pvec)$ and 
$\Hred(s,\pvec)$ by inspecting them on the imaginary axis; i.e., 
$\mid \Hfull(\imath \omega,\pvec) \mid$ vs $\mid \Hred(\imath \omega,\pvec)\mid $
where $\imath = \sqrt{-1}$ and 
$\omega \in \IR$.
The top plot in Figure \ref{figure Penzlp10} shows  the quality of the approximation for the fixed parameters  $(p_1,p_2,p_3)=(10,100,5000)$ and 
illustrates  that the  reduced model approximates the full  model very accurately. The relative error in the approximation, i.e.,
$\frac{\mid \Hfull(\imath \omega,\pvec) - \Hred(\imath \omega,\pvec)\mid}{
\max_{\omega}\mid \Hfull(\imath \omega,\pvec)\mid}$
is presented in the bottom plot of  Figure \ref{figure Penzlp10}. Note that this reduced model is obtained without any parameter sampling. 
\begin{figure}[htp]
\centering
			\includegraphics[width=0.7\paperwidth]{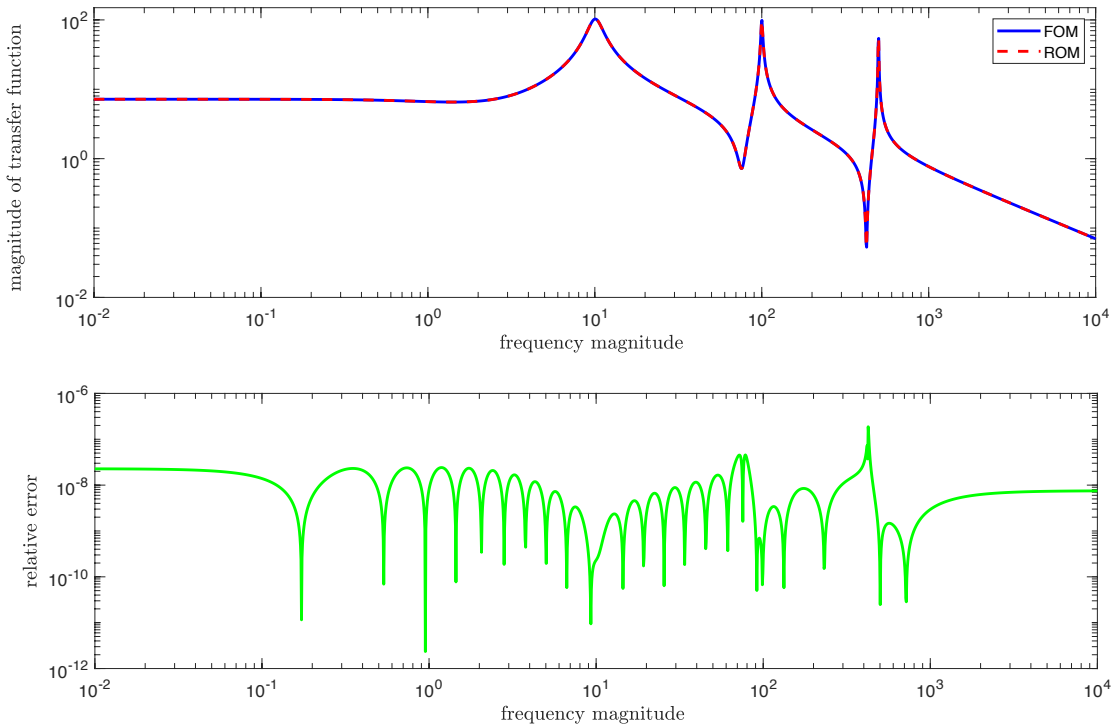}
            \caption{Transfer function plot and relative error for $(p_1,p_2,p_3)=(10,100,5000)$}\label{figure Penzlp10}
\end{figure}

\begin{figure}[htp]
\centering
			\includegraphics[width=0.7\paperwidth]{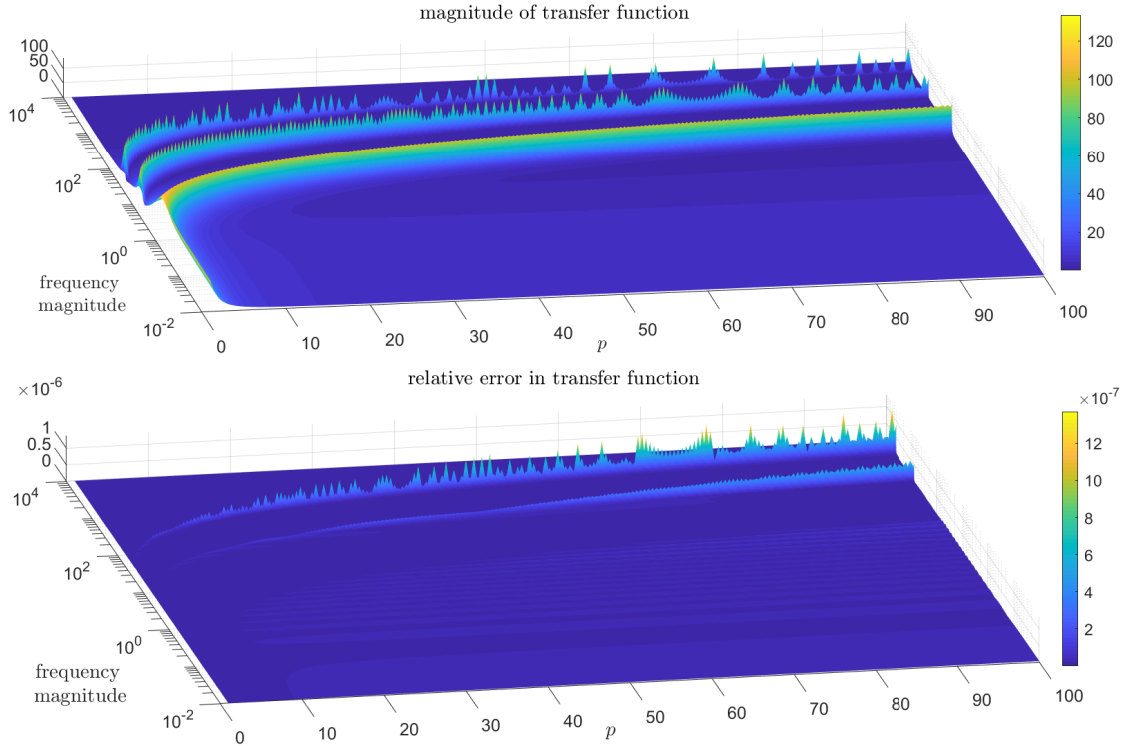}
            \caption{Magnitude of transfer function  and relative error for   parameters given by \eqref{parametersp1p2p3}}\label{figure Penzlsurf}
\end{figure}

It is not enough that $\Hred(s,\pvec)$ is accurate for one parameter set. In order to illustrate the quality of approximation and its ability to approximate the system for different  parameters,  we consider different configurations of parameters $p_1,p_2,p_3$. As mentioned earlier, the parameters control the imaginary part of the complex poles and different parameter selections will move these peaks in the frequency domain. In Figure \ref{figure Penzlsurf} we present the surface plot that illustrates
$\mid \Hfull(\imath \omega,\pvec) \mid$, as in Figure \ref{figure Penzlp10} but for different parameters $p_1,p_2,p_3$.
In order to obtain a three-dimensional surface plot, 
we choose 
\begin{equation}\label{parametersp1p2p3}
(p_1,p_2,p_3)=(\pvec,10\pvec,50\pvec), \quad \mbox{with}\quad   \pvec\in [1,100].
\end{equation}
On the top plot in Figure \ref{figure Penzlsurf} we show
$\mid \Hfull(\imath \omega,\pvec) \mid$, illustrating   how the peaks are moving with the parameter $p$. The lower subplot shows magnitude of  relative errors for  all considered parameters
and illustrates that the reduced model is accurate 
across the parameter domain \eqref{parametersp1p2p3}, with the largest relative error being less than $10^{-6}$. This accuracy is obtained by performing four non-parametric model reductions without any parameter sampling. We finally note that 
the reduced model  $\Hred(s,\pvec)$ is obtained to be asymptotically stable for every parameter sample. 
}
\end{example}  

\begin{example} { \em  In We  consider a model from the Oberwolfach Benchmark Collection representing thermal conduction in a semiconductor chip~\cite{morwiki_thermal}.
The full model is is described by
\begin{align*}
E \dot  x&=(A-p_t A_t-p_b A_b- p_s A_s) x +B u,\\
y &= C x ,
\end{align*}
where  $E \in   \IR^{4257\times 4257}$ represents the heat capacity and  $A\in \IR^{4257\times 4257}$  the heat conductivity. The matrices $A_t, A_b, A_s\in \IR^{4257\times 4257}$ are diagonal matrices resulting from the discretization of the convection boundary conditions with ranks  $111$, $99$, and $31$, respectively. The matrix  $ B\in   \IR^{1\times 4257}$ is the load vector and $C\in \IR^{7\times 4257}$ is the output matrix, while the parameters $p_t,p_b, p_s$ represent the film coefficients.
For further details for the model, we refer the reader 
to~\cite{morwiki_thermal,morFenRK04}.

We fix $p_t  = 1000$ and vary both $p_b$ and $p_s$ between $1$ and $10^9$; so here we have $\pvec=[p_b~ p_s]$. This model can be written in our structured parametric form as in \eqref{Apstructure} with $A_0=A-p_t A_t$, for the fixed $p_t=1000$, and $U$ and $V$ are matrices  with $\rank(U) = rank(V) = \rank(A_b)+\rank(A_s)=130$.
 
Based on the Hankel singular values, we reduce $\Hfull_i(s)$ to
$\Hred_i(s)$ via \textsf{BT}
using reduced orders 
$r_1=46,  r_2=66,  r_3=200,$  and $r_4=16$. The reduced order $r_3$ is bigger than the others as expected since it corresponds to approximating a dynamical system with $130$ inputs and $130$ outputs. 
In Figure \ref{errorThermalModel}, we illustrate the quality of the approximation obtained by Algorithm \ref{Method1} over the full parameter domain using the  $\htwo$-measure, i.e., 
$
\| H(\cdot,\pvec\|_{\htwo} = \sqrt{\frac{1}{2\pi}\int_{-\infty} 
     ^\infty \left\| \Hfull(\imath \omega,\pvec) \right\|_F^2\,d\omega}
$ (we will revisit system norms  in  Section \ref{ParameterOptimizatio}).
In Figure \ref{errorThermalModel},
the $x$ and $y$ axes represent, respectively,
the parameters $p_s$ and
$p_t$, and the  $z$-axis shows
the relative error 
$\frac{\|\Hfull(\cdot;  \pvec )-\Hred(\cdot; \pvec)\|_{\mathcal{H}_2}}{\|\Hred(\cdot; \pvec)\|_{\mathcal{H}_2}} $ for the reduced system  $\Hred(s,\pvec)$ calculated by Algorithm \ref{Method1}. It is clear from the figure that $\Hred(s,\pvec)$ is accurate across the full parameter domain, with relative error smaller than $2.4\times 10^{-6}$ for all parameters in $(p_s,p_b)\in [1, 10^9 ]^2$. This high-fidelity approximation is obtained  at the cost of four non-parametric subsystem model reduction without parameter sampling. As in the previous example, $\Hred(s,\pvec)$ is asymptotically stable for every parameter sample.

\begin{figure}[h]\begin{center}
 \input{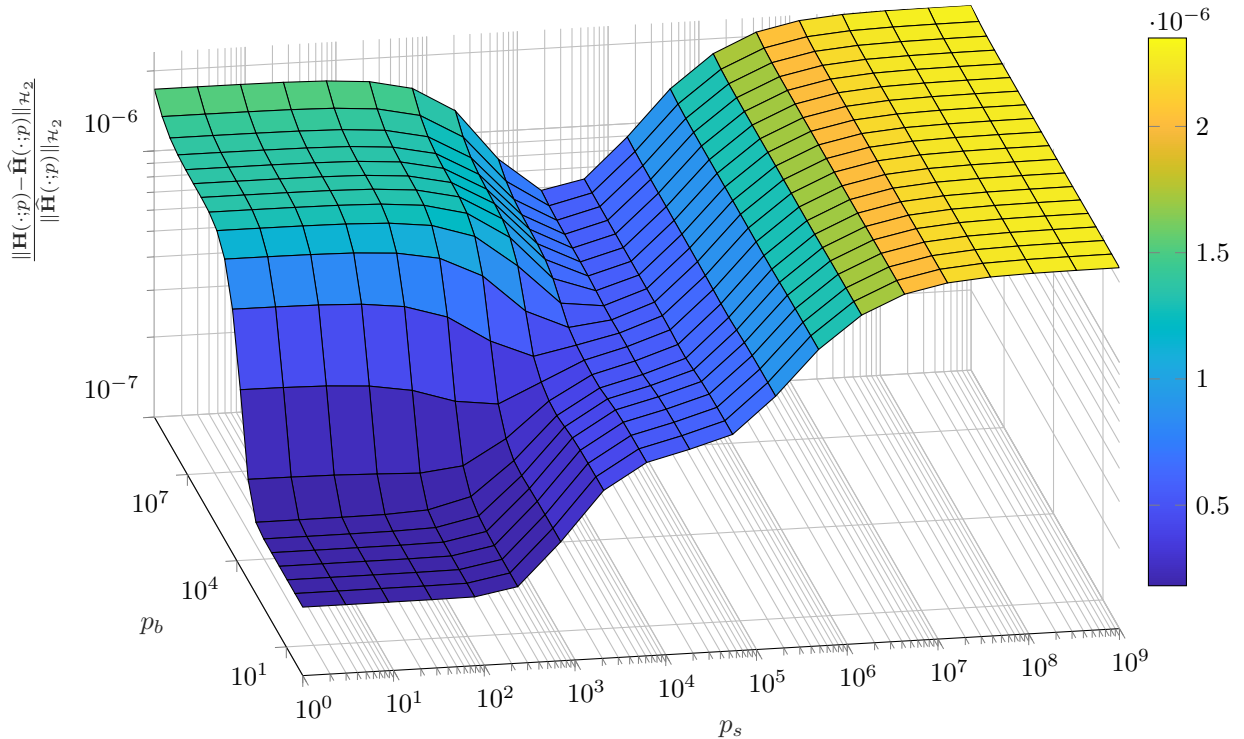}
 \caption{Relative error for different parameters for thermal model} \label{errorThermalModel}\end{center}
\end{figure}

}
\end{example}

\section{Data-driven PMOR with subsystem frequency sampling}  \label{sec:method2}

In Section \ref{sec:method1}, we proposed a sampling-free parametric model reduction approach that involved the preliminary reduction of four non-parametric models.  In this section we will present a second approach that  depends on the same four systems yet uses a data-driven framework based on transfer function (frequency-domain) samples to construct parametric reduced models in the offline stage. We will first briefly review data-driven modeling frameworks and then present the main approach.

\subsection{Data-driven modeling from transfer function samples}
 Let $\Hfull(s)$ denote the transfer function of a (nonparametric) linear dynamical system. $\Hfull(s)$ need not be a rational function of $s$ and can contain, for example, internal delays or other non-rational dependence in $s$.  Assume that we have access to samples of this  transfer function, i.e., we have $\Hfull(\xi_1),\Hfull(\xi_2),\ldots,\Hfull(\xi_N)$ where
 $\xi_i \in \IC$ for $i=1,2,\ldots,N$ are the sampling points. When obtained experimentally, these sampling points are chosen on the imaginary axis. If an analytical evaluation of  $\Hfull(s)$ is possible, they can be chosen arbitrarily as long as they do not coincide with the poles of $\Hfull(s)$. In our numerical experiments, we will work with samples on the imaginary axis but the theoretical discussion applies to the general case.
 
 Data-driven modeling in this case amounts to the following question: Given the samples $\{\Hfull(\xi_i)\}_{i=1}^N$ (without access to internal dynamics of $\Hfull(s)$, i.e., without access to a state-space transformation), construct a rational approximation $\Hred(s)$ of degree-$r$ that fits the data in an appropriate sense. There are various ways to fit this frequency domain data. One can enforce $\Hred(s)$ to interpolate the data at every sampling point using the Loewner framework \cite{mayo2007fsg,aca86},
or construct  $\Hred(s)$ to fit the data in a least-squares (LS) sense \cite{gustavsen1999rational,Drmac-Gugercin-Beattie:VF-2014-SISC,berljafa2017rkfit}, or force $\Hred(s)$ to interpolate some of the data and while minimizing the LS fit in the rest \cite{nakatsukasa2018aaa}. In this paper, we will fit data solely in a LS sense. 

Therefore, given the samples $\{\Hfull(\xi_i)\}_{i=1}^N$, our goal is to construct a degree-$r$ rational function
$\Hred(s)$, i.e., a reduced transfer function, that minimizes the LS error $\displaystyle{\sum_{i=1}^N \| \Hfull(\xi_i)-\Hred(\xi_i) \|_F^2}$.  Note that 
due to the nonlinear dependence on the poles of 
$\Hred(s)$, this is a nonlinear LS problem.
There are various approaches for solving this problem, see, e.g., 
\cite{hokanson2017projected,gonnet2011robust,gustavsen1999rational,Drmac-Gugercin-Beattie:VF-2014-SISC,nakatsukasa2018aaa,Sanathanan-Koerner-1963,berljafa2017rkfit}. Our approach  employs the Vector Fitting (\textsf{VF}) framework of \cite{gustavsen1999rational} even though one can easily adapt any of the other LS methods. We view \textsf{VF} as a tool  to be employed (as we did with \textsf{BT} and \textsf{IRKA} in subsystem reduction approach of Section \ref{sec:method1}) and therefore we do not explain it in detail. \textsf{VF} uses the  barycentric-form of $\Hred(s)$, as opposed to a state-space formulation, and  converts the nonlinear LS problem into a sequence of weighted linear LS problems each of which could be  solved easily by well-established numerical linear algebra tools in every step. The variables in each step are the coefficients of the barycentric form. Once the iteration is terminated, a state-space form is recovered. For details, we refer the reader to, for example, \cite{gustavsen1999rational,drmac2015vector,Gustavsen-2006,Chinea-G.Talocia-2011},
\cite[Chap 7]{grivet2015passive} and the references therein. 

We make a brief remark regarding  computational cost: \textsf{VF} performs
$m \cdot \ell$ QR factorizations of size
$N \times (2r)$  in every step \cite{drmac2015vector}. 
When $m$ and $\ell$ are modest, say $m,\ell < 10$, this is not a big computational effort. The cost will increase as $m$ and $\ell$ grow; however there are various ways to speed up the process such as performing the $m \cdot \ell$ QR factorizations in parallel \cite{Gustavsen-2006,Chinea-G.Talocia-2011} as they are independent of each other. We have not needed  any such sophisticated tools and a basic implementation proved efficient for us. Assume that 
the underlying system is a rational function itself; i.e., $\Hfull(s) = C(sE-A)^{-1}B$. Then, obtaining the samples $\{\Hfull(\xi_i)\}_{i=1}^N$ require solving $N$ linear systems of size $n \times n$ with multiple right-hand size; a much larger cost compared to \textsf{VF} itself. Therefore, the main cost of \textsf{VF} is indeed the sampling step itself.

\subsection{\textsf{pMOR} from offline samples}

Now, we discuss how we integrate the subsystem structure revealed in Proposition \ref{swmonH}
and \textsf{VF} for the parametric problems with the structured transfer function $\Hfull(s,\pvec)$.  As we briefly discussed above, 
the main cost in \textsf{VF} comes from computing the transfer function samples at selected frequencies. Therefore, we want to avoid re-sampling  $\Hfull(s,\pvec)$ from scratch for every given $\pvec$.

Recall \eqref{TFSMWstile}, which we repeat below
$$
\Hfull(s;\pvec)=\Hfull_1(s)- \Hfull_2(s) D(\pvec)  \left[I+ D(\pvec) \Hfull_3(s) D(\pvec)\right]^{-1} D(\pvec) \Hfull_4(s).
$$
Given the predetermined points  $\xi_1,\ldots, \xi_N$ in the complex plane,
compute the samples
$$\Hfull_1(\xi_i ) ,\, \Hfull_2(\xi_i) ,\,  \Hfull_3(\xi_i) ,\,  \Hfull_4(\xi_i) \quad \mbox{for}\quad  i=1,\ldots, N.$$
 It is important to note that these values do not depend on parameter $\pvec$, which means that we can perform this
computation once in the off-line stage. Furthermore, all four subsystem transfer functions share the same resolvent 
$(sE- A_0)^{-1}$; therefore in evaluating  $\Hfull_j(\xi_i )$ for $j=1,\ldots,n$, one takes advantage of this fact, significantly reducing the numerical cost of this step.

Then, for a parameter $\pvec$, using \eqref{TFSMWstile} we can efficiently calculate the values $\Hfull(\xi_i;\pvec)$  as 
\begin{align}\label{TFSMWstile xi}
\Hfull(\xi_i;\pvec)&=\Hfull_1(\xi_i)- \Hfull_2(\xi_i) D(\pvec)  (I+ D(\pvec) \Hfull_3(\xi_i) D(\pvec))^{-1} D(\pvec) \Hfull_4(\xi_i),
\end{align} 
for $i=1,\ldots,N$. This is step comes essentially at no cost. Therefore, we can re-sample 
$\Hfull(\xi_i;\pvec)$ for any $\pvec$ with almost no effort. Then, 
a data-driven approach, such as \textsf{VF} can be employed to construct a reduced model at a desired parameter value.  This is summarized in Algorithm \ref{Method2}.

\begin{algorithm}[htp] 
 \caption{Parametric reduced order model based on \textsf{VF}}   
\label{Method2}                                            
\begin{algorithmic} [1]                                        

\STATE  \label{1stII} \textsf{Off-line Stage:} For the predetermined points in the complex plane  $\xi_1,\ldots, \xi_N$ calculate
  $$\Hfull_1(\xi_i ) ,\, \Hfull_2(\xi_i) ,\,  \Hfull_3(\xi_i) ,\,  \Hfull_4(\xi_i) \quad \mbox{for}\quad  i=1,\ldots, N,$$
  using \eqref{H1-H4}.
\STATE \label{2ndII}  \textsf{On-line Stage:}\\ \qquad  For any given parameter $\pvec$ calculate $\Hfull(\xi_i;\pvec)$ for $i=1,\ldots, N$ using   formula \eqref{TFSMWstile xi}.
\STATE \qquad\label{3rdII}  Based on  $\Hfull(\xi_1;\pvec),\ldots, \Hfull(\xi_N;\pvec)$ obtain reduced system $\Hred(s;\pvec)$   using \textsf{VF}.
\end{algorithmic}
\end{algorithm}
For determining  the quality of the approximation resulting from Algorithm 
\ref{Method2}, we will use discrete LS error. That is, we will use the error measure calculated by
\begin{align}
\label{relativeerror}
 e(\Hfull(\cdot; \pvec) ),\Hred(\cdot;\pvec)))={\sum_{i=1}^{N}\left\|\Hfull(\xi_i;\pvec)-\Hred(\xi_i;\pvec) \right\|_F^2}.
\end{align}

Similar to Algorithm \ref{Method1}, sampling-based 
Algorithm \ref{Method2}
is well suited for computationally efficient parameter optimization and studying important system properties. 
In Algorithm \ref{Method2}, Step \ref{1stII} is executed only once in the off-line stage. Then, each time the parameter $\pvec$ is varied, (in the on-line stage), steps \ref{2ndII}-\ref{3rdII}  can be executed efficiently. 
In  the next section, we will present how one can use 
Algorithms \ref{Method1} and 
\ref{Method2}, and 
the error estimates given by \eqref{upperbound} and \eqref{relativeerror}, to   ensure  robust and accurate parameter optimization.  
\section{Parameter optimization for systems with low-rank parameterization}
\label{ParameterOptimizatio}

In this section we will present   algorithms for parameter optimization problems  that inherit the dynamical structure in \eqref{Apstructure}.  
We will incorporate the proposed sampling-free parametric model reduction techniques of Sections \ref{sec:method1} and \ref{sec:method2} into these optimization problems for efficient surrogate optimization.

Parameter optimization plays a vital role in many applications. In the case of damping optimization setting, this is computationally demanding problem even for moderate dimensions. The main reason lies in the fact that we need to optimize damping parameters (viscosities) together with damping positions that is a demanding combinatorial optimization problem. Optimization of damping parameters for the case of criteria based on system norms was studied, e.g., in \cite{Blanchini12, BennerKuerTomljTruh15, NTT19, TomljBeattieGugercin18}.  In this section we will present algorithms that allow parameter optimization in structured systems and in the section with numerical experiments we will apply these algorithms for efficient optimization of damping parameters. For
usage of model reduction in optimization in more general settings, we refer the reader to, e.g., 
\cite{BenSV14,Arian2002,yue2013,Antil2011,Kunisch2008,BuiThanh2008,alla2019certified,heinkenschloss2018reduced} and the references therein.

\subsection{The choice of cost function in damping optimization} 
Consider the following ODE-constrained optimization problem:
\begin{equation} \label{optprob1}
\begin{split}
   \pvec^{\star} &= \arg\min_{\pvec \in \Omega} \left\| y(\cdot,\pvec) \right\|
   \\[2ex]
\mbox{subject~to}\quad E  \dot x(t;\pvec) &= A(\pvec)x(t;\pvec) + Bu(t),\\
y(t;\pvec) &= Cx(t;\pvec).
\end{split}
\end{equation}
There are many viable choices for the norm selection $\|y(\cdot,\pvec)\|$ and 
the algorithms we describe below will apply to these various scenarios. However,
with the damping optimization problem of Section \ref{sec:motivexample} in mind
we will choose a specific norm discussed  below.

Recall that  in the damping optimization setting, the input $w$ represents an input disturbance and the goal is to minimize the influence of $w$  on the output $y$. Therefore, one might choose to minimize
$\| y(\cdot,\pvec) \|_{L_\infty} \coloneqq \sup_{t \geq 0} \| y(\cdot,\pvec) \|_{\infty} $
or 
$\| y(\cdot,\pvec) \|_{L_2} \coloneqq \sqrt{\int_0^\infty \| y(\cdot,\pvec) \|_{2}^2 dt}.$
These norms can be equivalently represented using the transfer function $\Hfull(s,p)$. The corresponding frequency-domain norms are the $\htwo$ and $\hinf$ norms: 
\begin{equation}
    \left\| \Hfull(\cdot,\pvec) \right\|_{\htwo}
     \coloneqq 
     \sqrt{\frac{1}{2\pi}\int_{-\infty} 
     ^\infty \left\| \Hfull(\imath \omega,\pvec) \right\|_F^2\,d\omega} 
    \quad \mbox{and} \quad
     \left\| \Hfull(\cdot,\pvec) \right\|_{\hinf}
    \coloneqq \sup_{\omega \in \IR} 
 \left\| \Hfull(\imath \omega,\pvec) \right\|_2,
\end{equation}
where $\imath^2 = -1$ and $\| \cdot \|_F$ denotes the Frobenius norm. For a stable linear
(parametric) dynamical systems with an input $w(t)$ having 
$\| w  \|_{L_2}\leq \infty$  and the corresponding output $y(t;\pvec)$, it holds
$$
\| y(\cdot,\pvec) \|_{L_\infty}
\leq   \left\| \Hfull(\cdot,\pvec) \right\|_{\htwo} \| w  \|_{L_2}
\quad \mbox{and}\quad 
\| y(\cdot,\pvec) \|_{L_2}
\leq   \left\| \Hfull(\cdot,\pvec) \right\|_{\hinf} \| w  \|_{L_2}.
$$
Therefore the optimization problem \eqref{optprob1}, with the choice of the $L_\infty$ norm,  can be equivalently rewritten as 
\begin{equation}
\label{ouroptprob}
    \pvec^{\star} = \arg\min_{\pvec \in \Omega} \left\| \Hfull(\cdot,\pvec) \right\|_{\htwo}
    ~~~\mbox{where}~~~
    \Hfull(s,\pvec) = C \left(s E - \left( A_0-U\,D^2(\pvec)V^T\right)\right)^{-1}B.
\end{equation}
In the discussion below, we will present the analysis and algorithms for the parameter optimization problem \eqref{ouroptprob}.

\subsection{Surrogate optimization with reduced parametric  models} 
The major cost in \eqref{ouroptprob} is the computation of  the $\htwo$ norm. For 
$ \Hfull(s;\pvec) = C(sE-A(p))^{-1}B$, the $\htwo$ norm is computed by solving a Lyapunov equation:
$$
\left\| \Hfull(\cdot,\pvec) \right\|_{\htwo} = \sqrt{ 
\textsf{trace}(C P C^T)} \quad \mbox{where~$P$~solves}\quad
 A(\pvec) P E^T + E P A(\pvec)^T + B B^T = 0.
$$
Solving a large-scale Lyapunov equation is computationally demanding and in this optimization setting one has to repeat this task for many different $\pvec$ values. We will use the parametric reduced models from Algorithms \ref{Method1} and \ref{Method2} to relieve this computational burden. Therefore, as opposed to
\eqref{ouroptprob}, we will solve the surrogate optimization problem
\begin{equation}
\label{suroptprob}
    \hat{\pvec}^{\star} = \arg\min_{\pvec \in \Omega} \left\| \Hred(\cdot,\pvec) \right\|_{\htwo},
\end{equation}
where the reduced parametric transfer function $\Hred(\cdot,\pvec)$
will be constructed as in   either Algorithm \ref{Method1}
or Algorithm \ref{Method2}, without need for parameter sampling.

Assume $\pvec^\star$ is the minimizer of \eqref{ouroptprob} and note that 
\begin{equation} \label{optupperbound}
\left \| \Hfull (\cdot,\pvec^\star) \right \|_\htwo \leq
\left\| \Hfull(\cdot,\pvec^\star) - 
\Hred (\cdot,\pvec^\star) \right\|_\htwo
+ \left\| \Hred (\cdot,\pvec^\star) \right\|_\htwo.
\end{equation}
The surrogate optimization problem \eqref{suroptprob} will minimize the second term in \eqref{optupperbound}. Therefore, we need to verify that the first term in \eqref{optupperbound}
is small enough. In other words, we need the reduced model $\Hred(s,\pvec)$ to be an  accurate approximation at the minimizer $\hat{\pvec}^\star$.

\subsubsection{Surrogate optimization with reduced model  via Algorithm \ref{Method1}} 
\label{suroptmethod1}
To guarantee that $\Hred(s,\pvec)$ is accurate enough at the optimizer $\hat{\pvec}^\star$, we need to be evaluate the term
$\left\| \Hfull(\cdot,\hat{\pvec}^\star) - 
\Hred (\cdot,\hat{\pvec}^\star) \right\|_\htwo$. Therefore, an efficient evaluation (estimation) of this term 
during optimization is 
 crucial for a numerically effective implementation. 
 
When Algorithm \ref{Method1} is employed to construct the reduced model $\Hred(s,p)$,
Theorem \ref{upper bound}, more specifically \eqref{upperbound}, 
shows how $\left\| \Hfull(\cdot,\hat{\pvec}^\star) - 
\Hred (\cdot,\hat{\pvec}^\star) \right\|_\htwo$  can be
bounded using the subsystem errors 
 $\epsilon_i = \|\Hfull_i(\cdot) - \Hred_i(\cdot)\|$, for $i=1,2,3,4$. Unfortunately, two of the terms in
 \eqref{upperbound} depend on the full order quantities $\Hfull_1(s)$
 and $\Hfull_2(s)$. Therefore, in our surrogate optimization routine, we will use an approximation to this upper bound. Assuming $\Hred_2(s)$ and $\Hred_3(s)$  are accurate approximations to  $\Hfull_2(s)$ and $\Hfull_3(s)$,
 i.e., $\Hfull_2  \approx\Hred_2 $, and $  \Hfull_3  \approx\Hred_3 $(note that we can control this accuracy in the model reduction stage), 
 using  Theorem \ref{upper bound}, we will approximate the upper bound as 
\begin{align} 
 \|\Hfull(\cdot; \pvec) \nonumber )-\Hred(\cdot;\pvec))\| \lesssim   \epsilon_1 & +\epsilon_2  f_1(\pvec,\Hred_3,\Hred_4) \\ 
 & + \epsilon_3 f_1(\pvec,\Hred_3,\Hred_4) f_2(\pvec,\Hred_2 ,\Hred_3   )+ \epsilon_4 f_2(\pvec,\Hred_2 ,\Hred_3   ),\label{upperboundf_i}
\end{align}
 where functions $f_1$ and $f_2$ and given by \eqref{DefFunf1} and \eqref{DefFunf2}, respectively,
 and $\epsilon_i = \|\Hfull_i(\cdot) - \Hred_i(\cdot)\|$
 for $i=1,\ldots,4.$ To simplify the notation, we will denote the upper bound estimate, i.e., the right hand-side of \eqref{upperboundf_i}, with $f(\pvec)$.

Now, using $f(\pvec)$,  we can efficiently estimate the accuracy of the 
reduced model at a given parameter value $\pvec$.  In Algorithm 
\ref{parameteroptimizationwithmethod1},
we give an outline of a surrogate optimization method using this estimate. 
Starting with initial reduced subsystems in  Step \ref{comproms} (constructed for a given accuracy),  Algorithm 
\ref{parameteroptimizationwithmethod1} solves the surrogate optimization problem  in Step \ref{suropt1inalg}.
Then, Step \ref{whilestep} checks whether the reduced model is accurate enough at the current optimizer using the estimate  $f(p)$ for the upper bound. If it is, the algorithm terminates. Otherwise, Step 
\ref{increaseorder} adaptively increases the reduced dimension and Step \ref{optp} resolves the surrogate optimization for the updated reduced model. This procedure is repeated until desired tolerance is met.

 \begin{algorithm}[htp] 
 \caption{Surrogate parameter optimization using reduced models via Algorithm 
 \ref{Method1}}   
\label{parameteroptimizationwithmethod1}                                            
\begin{algorithmic} [1]                                        
\REQUIRE  System matrices $A(\pvec),E,B,C$ defining
\eqref{Mainsystem};\\
initial point $\pvec^0$ for optimization routine\\
tolerance $0<\tau\ll 1$ for error bound\\
tolerance $0<\nu\ll 1$ for optimization routine\\
starting reduced dimensions $r_1,r_2,r_3,r_4$ and the corresponding subsystems
$\Hred_1$, $\Hred_2$, $\Hred_3$, and $\Hred_4$. 
 \ENSURE  approximation of optimal parameters
\STATE \label{comproms}   Choose the reduced orders $r_1,r_2,r_3$, $r_4$ (and thus the reduced subsystems $\Hred_1(s)$,
$\Hred_2(s)$, $\Hred_3(s)$, and $\Hred_4(s)$ via Algorithm \ref{Method1}) so that  $f(\pvec^0)< \tau$.
\STATE \label{suropt1inalg} Solve the surrogate optimization problem 
$$\hat{\pvec}^\star = \arg\min_{\pvec} \left\| \Hred(\cdot,\pvec) \right\|_{\mathcal{H}_2}  $$
with the initial guess $\pvec^{0}$ and 
tolerance $\nu$. 
\WHILE 
{minimizer $\pvec^\star$ such that $f(\pvec^\star)>\tau$} 
\label{whilestep}
\STATE  $\pvec^0=\hat{\pvec}^\star$
\STATE \label{increaseorder}
Increase the reduced orders $r_1,r_2,r_3$, $r_4$ (and thus the reduced subsystems  via Algorithm \ref{Method1}) so that  $f(\hat{\pvec}^\star)< \tau$.
\STATE \label{optp} Determine the new minimizer 
by solving the surrogate optimization problem 
$$\hat{\pvec}^\star = \arg\min_{\pvec} \left\| \Hred(\cdot,\pvec) \right\|_{\mathcal{H}_2}  $$
using the updated $\Hred$, the initial guess $\pvec^{0}$, and tolerance $\nu$. 
\ENDWHILE
\end{algorithmic}
\end{algorithm}

There are various algorithmic details 
that will  help speed up the computations. We will not dive into those details; instead highlight some points. For example, assume that subsystem model reduction in Algorithm \ref{Method1} is performed using \textsf{BT}. Then, to increase the reduced dimensions in Step \ref{increaseorder}, one does not need to apply  model reduction  from scratch. In the case of \textsf{BT}, one will only need to add more vectors to the \textsf{BT}-based model reduction bases from already computed quantities. If \textsf{IRKA} is employed in Algorithm \ref{Method1}, then the current reduced-model poles, appended with some others, will 
be an  effective initialization strategy for \textsf{IRKA}, yielding faster convergence.

\subsubsection{Surrogate optimization with reduced model  via Algorithm \ref{Method2}} 

We now focus on solving the optimization problem \eqref{ouroptprob} using 
reduced models  from the data-driven Algorithm \ref{Method2}. One major difference from 
Algorithm \ref{Method1} is that in this case there are no reduced subsystems. Instead, for every given $\pvec$, we have a numerically efficient way to find 
an accurate approximation $\Hred(s,\pvec)$ to
$\Hfull(s,\pvec)$. Therefore, the subsystem-based  error estimate in \eqref{upperboundf_i} does not apply here. However, we have the sample-based (discretized) version
$e(\pvec)$ defined in 
 \eqref{relativeerror}. For example, if \textsf{VF} is employed in Algorithm \ref{Method2}, the error $e(\pvec)$ will be automatically calculated during the construction of $\Hred(s,\pvec)$ and thus no additional effort is needed  to compute $e(\pvec)$. Therefore, following similar arguments to those found in Section \ref{suroptmethod1},
 in solving the surrogate optimization problem \eqref{suroptprob}, 
 we need to ensure that the reduced model $\Hred(s,\pvec)$ is an accurate approximation to $\Hfull(s,\pvec)$ at the optimizer $\pvec = \hat{\pvec}^\star$ where accuracy is now measured using $e(p)$.
 
 The resulting approach is briefly discussed in Algorithm 
 \ref{parameteroptimizationwithmethod2}. The fundamental structure is almost identical to that of Algorithm \ref{parameteroptimizationwithmethod1}. We test whether $e(\hat{\pvec}^\star)$ is below a  prespecified tolerance. If not, we increase the order of the reduced model in Algorithm 
 \ref{Method2} until we reach the desired accuracy.

\begin{algorithm}[htp] 
 \caption{Surrogate parameter optimization using reduced models via Algorithm 
 \ref{Method2}}   
\label{parameteroptimizationwithmethod2}                                            
\begin{algorithmic} [1]                                        
\REQUIRE  System matrices $A(\pvec),E,B,C$ defining
\eqref{Mainsystem};\\
initial point $\pvec^0$ for optimization routine\\
tolerance $0<\tau\ll 1$ for error bound\\
tolerance $0<\nu\ll 1$ for optimization routine\\
Samples $\{\Hfull_i(\xi_i)\}_{i=1}^N$ at predetermined points in the complex plane  $\xi_1,\ldots, \xi_N$.
 \ENSURE  approximation of optimal parameters
\STATE {Choose the reduced order  in Algorithm \ref{Method2} so that $e(\pvec^0) < \tau$.}
\STATE \label{suropt1inalg2} Solve the surrogate optimization problem 
$$\hat{\pvec}^\star = \arg\min_{\pvec} \left\| \Hred(\cdot,\pvec) \right\|_{\mathcal{H}_2}  $$
with the initial guess $\pvec^{0}$ and 
tolerance $\nu$ with $\Hred_{\pvec}$ computed via
Algorithm \ref{Method2} using samples
$\{\Hfull_i(\xi_i)\}_{i=1}^N$.
\WHILE 
{minimizer $\pvec^\star$ such that $e(\pvec^\star)>\tau$} 
\label{whilestep2}
\STATE  $\pvec^0=\hat{\pvec}^\star$
\STATE \label{increaseorder2}
Increase the reduced order used in Algorithm \ref{Method2} so that $e(\pvec^\star) < \tau$.
\STATE \label{optp2} Determine the new minimizer 
by solving the surrogate optimization problem 
$$\hat{\pvec}^\star = \arg\min_{\pvec} \left\| \Hred(\cdot,\pvec) \right\|_{\mathcal{H}_2}  $$
using the updated $\Hred$, the initial guess $\pvec^{0}$, and tolerance $\nu$.
\ENDWHILE
\end{algorithmic}
\end{algorithm}

As is the case with Algorithm 
\ref{parameteroptimizationwithmethod1},
there are various numerical aspects that one could exploit to make the online computations faster. For example, one main factor determining the convergence speed of \textsf{VF} is an initial selection of poles. In the damping optimization problem,  the poles from the critical damping case, are perfect candidates. Also, if one has to increase the order in Step \ref{increaseorder2} of Algorithm \ref{parameteroptimizationwithmethod2},
the already-converged poles  from the previous optimization step (appended with a small number of additional ones) is certainly expected to speed up convergence. We will elaborate on these points in the numerical example below.

\subsection{Numerical Example}

We revisit the damping optimization problem described in Section \ref{sec:motivexample} and focus on optimizing the viscosities for different sets of damping positions. We consider an  $n$-mass oscillator with $n=2d+1$ masses and $2d+3$ springs as shown in Figure \ref{FigOscillator}. This oscillator has  two rows of $d$ masses connected with springs. The leading masses in each row on the left edge   are connected to  a fixed boundary while  on the opposite (right) edge the masses ($m_d$ and $m_{2d}$)  are connected to a single  mass $m_{2d+1}$, which, in turn, is connected to a fixed boundary. See \cite{BennerKuerTomljTruh15, TomljBeattieGugercin18} for further details.
\begin{figure}[h]
\begin{center}
 \includegraphics[width=9cm]{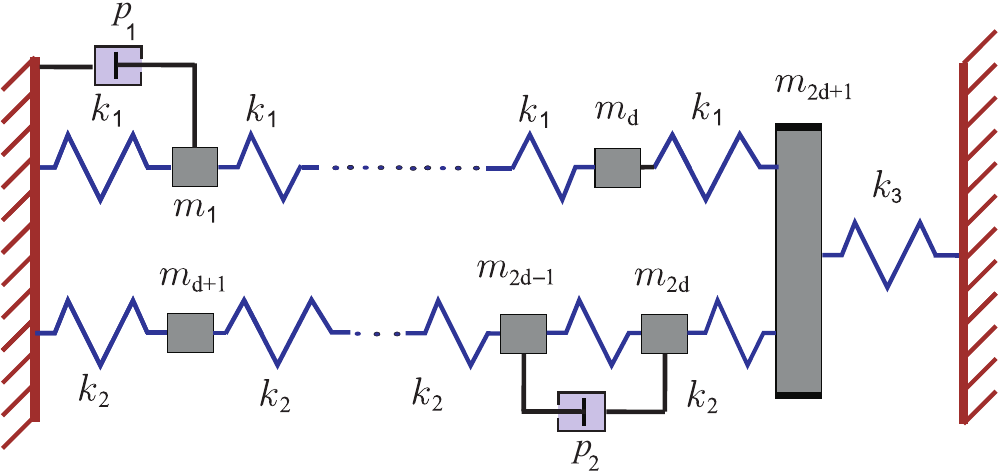}
 \end{center}
\label{FigOscillator}
\end{figure}

The state-space model is given by \eqref{MDK} where the   stiffness matrix is given by
\begin{align*}
K =\begin{bmatrix}
    K_{11} &   &  -\kappa_1\\[2mm]
      & K_{22} &  -\kappa_2\\[2mm]
     -\kappa_1^T & -\kappa_2^T  &  k_1+k_2+k_3\\
 \end{bmatrix},\, \,
 K_{ii} =k_i
 \begin{bmatrix}
    2 & -1  & & &\\
     -1 & 2 &-1 & & \\
      &     \ddots & \ddots  & \ddots&\\
      &    & -1 & 2 & -1 \\
       & & &-1 & 2\\
 \end{bmatrix},
\end{align*}
with $\kappa_i=\begin{bmatrix}0&
 \ldots&
 0&
 k_i
 \end{bmatrix} $ for $i=1$ and $i=2$, and the mass matrix is $M  =\diag{(m_1,m_2,\ldots,m_n)} $.

 We pick $n=1801$ masses ($d=900$)  with the  values
\begin{align*}
  & \quad m_i= \left\{ \begin{array}{ll} 1000-  \frac{i}{2}, \quad & i=1,\ldots,450, \\
   i+325  , \quad & i=451,\ldots,900, \\
  1300  -  \frac{i}{4}  , \quad & i=901,\ldots,n.
                 \end{array}\right.
\end{align*}
The  stiffness values are chosen as
$  k_1=500, k_2=200$, and $k_3=300$. 
The parameter $\alpha_c$ that determines the influence of the internal damping defined by  \eqref{C_{int}} is set to $ 0.02$. The primary excitation corresponds to five masses closest to ground, i.e., $B_2\in \IR^{n\times 5}$ with
 \begin{align*}
  B_2(1:2,1:2) &=\diag(20,10),\\
    B_2(901:902,3:4) &=\diag(20,10),\\
    B_2(1801,5)&=30;
 \end{align*}
and all other entries being zero.
We are interested in  
the two displacements, yielding the output 
 $$ y(t;p)=\begin{bmatrix} q_{400}(t;p)& q_{1300}(t;p) \end{bmatrix}^T.$$

In this example we consider 
optimization over four dampers with gains $p_1, p_2, p_3$ and $p_4$ with their positions encoded in
\begin{equation*}
 U_2  =\left [\,e_{j_1}-e_{j_1+10}, ~~ e_{j_2}, ~~ e_{j_3},  ~~ e_{j_3}-e_{j_3+100} \,\right],
\end{equation*}
where $e_i$ is the $i$th canonical vector and  the indices $j_1, j_2, j_3$ determine the damping positions.
In order to illustrate the performance of  our surrogate optimization framework for different damping configurations,  the following indices   are considered:   $j_1 \in \{100,300,500,700\}$,  $j_2\in\{ 150,   350,   550,  750\}$, and $j_3\in\{ 1400, 1700\}$. This results in $32$ different damping configurations for which we optimize   $\mathcal{H}_2$ system norm, i.e., we solve \eqref{ouroptprob} and the surrogate problem 
\eqref{suroptprob}.
 
The full optimization problem
\eqref{ouroptprob} and the surrogate problem \eqref{suroptprob}
were solved using   \textsc{Matlab}'s built-in \verb"fminsearch" together with a  transformation that allows  constrained optimization. The starting point was $\pvec^0=(100,100, 100, 100)$ and for each parameter, the range was $[0,5000]$. The stopping tolerance for optimization was set to $\nu=5\cdot 10^{-4}$. In solving the surrogate optimization problem, we employed both Algorithms \ref{parameteroptimizationwithmethod1} and 
\ref{parameteroptimizationwithmethod2}.
Our implementation will take advantage of the fact that the first subsystem $\Hfull_1(s)$ is 
independent of not only the parameter $\pvec$ but also the damping positions. Therefore, for   the $32$ damping configurations considered, reducing   
$\Hfull_1(s)$ in Algorithm \ref{parameteroptimizationwithmethod1} (or sampling of $\Hfull_1(s)$ in Algorithm \ref{parameteroptimizationwithmethod2} needs to be done only once.

In  Algorithm \ref{parameteroptimizationwithmethod1} for each damping configuration, we used
\begin{align*}
 \quad \mbox{ termination tolerance for error bound:\quad} \tau =10^{-2};   \\
 \mbox{initial reduction dimensions:\quad} (r_1,r_2,r_3,r_4)=(280, 300, 480, 430).
 \end{align*}
 The reduced orders $r_i$ were chosen based on the Hankel singular values of each subsystem. 
Reduced subsystem updates were performed such that each time an update was needed,  $r_i$ was increased by 15\%. 
Similarly, in  Algorithm 
\ref{parameteroptimizationwithmethod2} for each damping configuration, we used:
\begin{align*}
 \quad \mbox{ termination tolerance for error bound\quad} \tau =10^{-4};   \\
 \mbox{initial reduced-order was set to\quad} 130;\\
 \mbox{number of  predetermined sampling points  $\xi_1,\ldots, \xi_N$  was set to\quad} N=500. 
 \end{align*}
The sampling points $\xi_1,\ldots, \xi_N$ were chosen to be logarithmically spaced between the smallest and largest (by magnitude) undamped eigenfrequencies. 
We initialized \textsf{VF} using 
dominant poles \cite{BennerKuerTomljTruh15}. During the optimization process each time that  $ e(\Hfull(\cdot; \pvec) ),\Hred(\cdot;\pvec)))>\tau$, the order  was increased by 10\%. 

 Across all damping configurations,  Algorithms \ref{parameteroptimizationwithmethod1} entered  the inner while loop 
only in 15\% of the cases
and Algorithm \ref{parameteroptimizationwithmethod2} in 34\% of the cases. For Algorithm \ref{parameteroptimizationwithmethod1}, this means that model reduction was performed only once for most configurations.
For Algorithm \ref{parameteroptimizationwithmethod2}, note that it employs Algorithm \ref{Method2}, which resamples $\Hfull(s,\pvec)$ almost at no cost, and then applies \textsf{VF}.  Repeated application of \textsf{VF} constitutes only a modest cost increment since the required frequency sampling is obtained already at an earlier step.
  Remarkably, we have also observed that in nearly all cases that we consider, \textsf{VF} converges quickly. In particular, across all $32$ damping configurations, application of \textsf{VF} for the initial surrogate optimization step took on average 13 iterations to converge. But in the vast majority of subsequent \textsf{VF} application, convergence occurred often after a single iteration (93.5\% of total \textsf{VF} applications) or two iterations (3.3\% of total \textsf{VF} applications).

Figure \ref{RelErrGains} depicts the relative errors in the optimal gains for different damping configurations calculated by Algorithm \ref{parameteroptimizationwithmethod1} 
(denoted by blue squares)
and Algorithm \ref{parameteroptimizationwithmethod2}
(denoted by black triangles).
The relative errors in the optimal gain is calculated by $\|\pvec^\star-\hat{\pvec}^\star\|/\|\pvec^\star\|$, where $\pvec^\star$ and
$\hat{\pvec}^\star$ denote the optimal gains calculated with, respectively, the full model 
(i.e., solving \eqref{ouroptprob}) and 
the reduced model (i.e., solving \eqref{suroptprob}). 
The figure shows in most cases Algorithm \ref{parameteroptimizationwithmethod2} 
gave more accurate results.

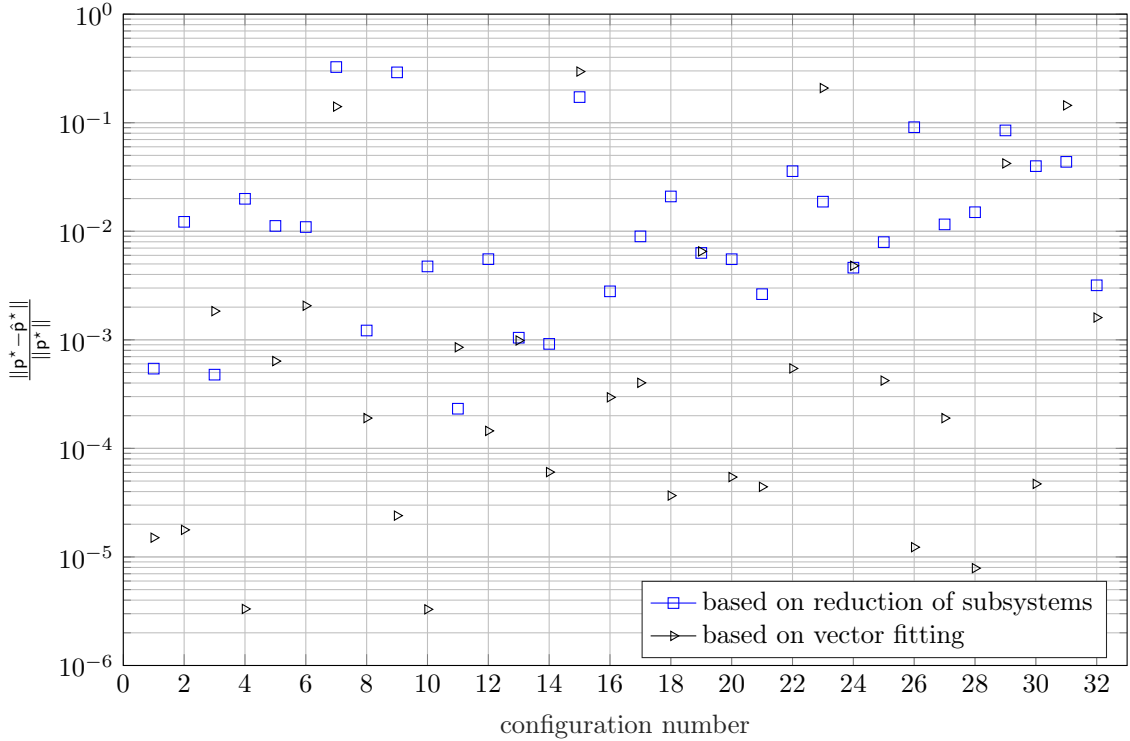
\begin{figure}[h]\begin{center}
%
%
\begin{tikzpicture}

\begin{axis}[%
width=5.2in,
height=3.4in,
at={(0.758in,0.481in)},
scale only axis,
xmin=0,
xmax=33,
xlabel style={font=\color{white!15!black}},
xlabel={configuration number},
ylabel={$\frac{\|\pvec^{\star}-\hat{\pvec}^\star\|}{\|\pvec^{\star}\|}$},
ymode=log,
ymin=1e-06,
ymax=1,
yminorticks=true,
axis background/.style={fill=white},
xmajorgrids,
ymajorgrids,
yminorgrids,
legend style={at={(0.98,0.13)}, legend cell align=left, align=left, draw=white!15!black}
]
\addplot [color=blue, draw=none, mark=square, mark options={solid, blue}]
  table[row sep=crcr]{%
1	0.000542720652483514\\
2	0.0122146171072611\\
3	0.000477600761666868\\
4	0.0199074741507212\\
5	0.0111981451441596\\
6	0.0109384230262138\\
7	0.325644291026748\\
8	0.00121842428255113\\
9	0.290865223373094\\
10	0.00474379801165942\\
11	0.000231868989147354\\
12	0.00553501887390266\\
13	0.00104575334429056\\
14	0.000917171946058492\\
15	0.172207997605939\\
16	0.00279661365873561\\
17	0.00897325546759503\\
18	0.0209244372064201\\
19	0.00631430898302534\\
20	0.0055300515551673\\
21	0.00263409166444679\\
22	0.0357800315084178\\
23	0.0187379620064519\\
24	0.0046189801283156\\
25	0.00794364609560898\\
26	0.0909288480826826\\
27	0.0115680051346244\\
28	0.0149809146147973\\
29	0.0849092215483358\\
30	0.039814538307721\\
31	0.0435915246933554\\
32	0.00317875053471998\\
};
\addlegendentry{based on reduction of subsystems}

\addplot [color=black, draw=none, mark=triangle, mark options={solid, rotate=270, black}]
  table[row sep=crcr]{%
1	1.50418649871178e-05\\
2	1.77502518542077e-05\\
3	0.00183760518162192\\
4	3.3188868676057e-06\\
5	0.000636906065560952\\
6	0.00206050104555412\\
7	0.140673581750937\\
8	0.00019015496834216\\
9	2.40110105443066e-05\\
10	3.29153336253174e-06\\
11	0.000853503410901402\\
12	0.000144902474423027\\
13	0.000985835732440148\\
14	6.04468834222777e-05\\
15	0.295943512270338\\
16	0.000295321794239509\\
17	0.000401672371874046\\
18	3.67371843430772e-05\\
19	0.00652623448654331\\
20	5.44203709940598e-05\\
21	4.42463955889151e-05\\
22	0.000545090012080097\\
23	0.208447110551895\\
24	0.00480968687609478\\
25	0.000420994654532263\\
26	1.22991590559305e-05\\
27	0.000190026083831126\\
28	7.88022620187967e-06\\
29	0.0421629351828864\\
30	4.7090256744241e-05\\
31	0.14438417240318\\
32	0.00160284560089078\\
};
\addlegendentry{based on vector fitting}

\end{axis}
\end{tikzpicture}%
 \caption{Relative errors in the optimal gains  for Algorithm \ref{parameteroptimizationwithmethod1}  and Algorithm 
 \ref{parameteroptimizationwithmethod2}} \label{RelErrGains} \end{center}
\end{figure}

However, since the cost function, the $\mathcal{H}_2$ norm, can be flat with respect to some damping parameters, the quality of the surrogate optimization is better illustrated in Figure \ref{RelErrH2}  where we show  the relative errors in the the cost function. For the optimal gain $\pvec^\star$ obtained by solving the full problem \eqref{ouroptprob}
and the optimal gain $\hat{\pvec}^\star$
obtained by solving the surrogate problem \eqref{suroptprob},
the relative error is computed by ${ \frac{\Large| \|\Hfull(\cdot;  \pvec^{\star} )\|_{\mathcal{H}_2}-\|\Hred(\cdot; \pvec^\star)\|_{\mathcal{H}_2}\Large|}{\|\Hred(\cdot; \pvec^{\star})\|_{\mathcal{H}_2}} }$. These results are illustrated in In Figure \ref{RelErrH2}.
Even though both algorithms yield accurate results, the surrogate optimization with Algorithm \ref{parameteroptimizationwithmethod2} is consistently better with the largest relative error in the order of $10^{-4}$.

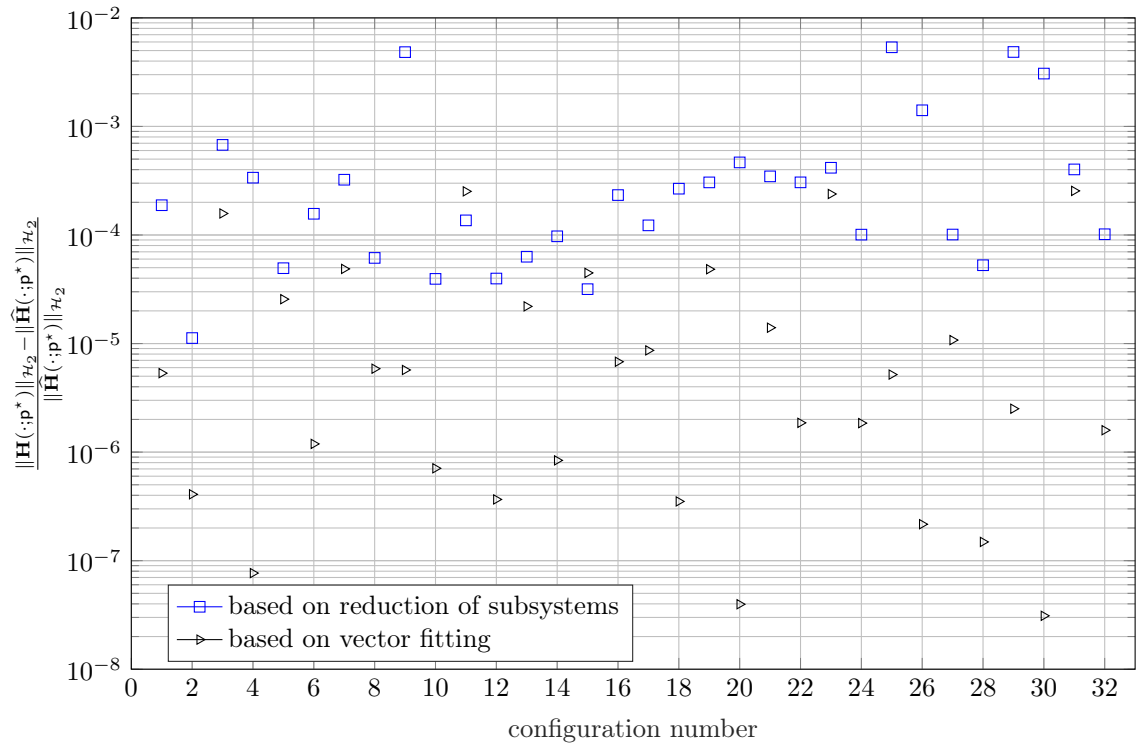
\begin{figure}[h]\begin{center}
%
%
\begin{tikzpicture}

\begin{axis}[%
width=5.2in,
height=3.4in,
at={(0.758in,0.481in)},
scale only axis,
xmin=0,
xmax=33,
xlabel style={font=\color{white!15!black}},
xlabel={configuration number},
ylabel={$\frac{\|\mathbf{H}(\cdot;  \pvec^{\star} )\|_{\mathcal{H}_2}-\|\mathbf{\widehat H}(\cdot; \pvec^{\star})\|_{\mathcal{H}_2}}{\|\mathbf{\widehat H}(\cdot; \pvec^{\star})\|_{\mathcal{H}_2}} $},
ymode=log,
ymin=1e-08,
ymax=0.01,
yminorticks=true,
axis background/.style={fill=white},
xmajorgrids,
ymajorgrids,
yminorgrids,
legend style={at={(0.5,0.13)}, legend cell align=left, align=left, draw=white!15!black}
]
\addplot [color=blue, draw=none, mark=square, mark options={solid, blue}]
  table[row sep=crcr]{%
1	0.000188125354508749\\
2	1.1238775833773e-05\\
3	0.000676257323212229\\
4	0.000337313893024712\\
5	4.94627247437791e-05\\
6	0.000156850506806493\\
7	0.000322982228306512\\
8	6.1493019388938e-05\\
9	0.00483705210801925\\
10	3.9425513597948e-05\\
11	0.00013637606238636\\
12	3.97204170823388e-05\\
13	6.30834369211221e-05\\
14	9.72676042702569e-05\\
15	3.17010044024441e-05\\
16	0.000233264554337414\\
17	0.000122695634239343\\
18	0.000267061203549691\\
19	0.000305236860006198\\
20	0.000466339310654872\\
21	0.000346777169486422\\
22	0.000305503189601267\\
23	0.000416263019831142\\
24	0.000100774155853269\\
25	0.00536152885120921\\
26	0.00140867025557642\\
27	0.000100972010611159\\
28	5.26736253175725e-05\\
29	0.00485179073804885\\
30	0.00306926983116812\\
31	0.000401891732012514\\
32	0.00010162619388106\\
};
\addlegendentry{{based on reduction of subsystems}}

\addplot [color=black, draw=none, mark=triangle, mark options={solid, rotate=270, black}]
  table[row sep=crcr]{%
1	5.33563506538586e-06\\
2	4.0856303463607e-07\\
3	0.000157864081517049\\
4	7.68479173662497e-08\\
5	2.56331654282833e-05\\
6	1.19024986845655e-06\\
7	4.87127622268622e-05\\
8	5.86572637741794e-06\\
9	5.71010777847803e-06\\
10	7.09529383541401e-07\\
11	0.000252342792067439\\
12	3.66522730070886e-07\\
13	2.19986474251149e-05\\
14	8.39317553099207e-07\\
15	4.47532666437398e-05\\
16	6.80058157042293e-06\\
17	8.65182652411145e-06\\
18	3.51431149825323e-07\\
19	4.8385130674129e-05\\
20	3.9685116710884e-08\\
21	1.3996916955602e-05\\
22	1.86084557258651e-06\\
23	0.000238678449420641\\
24	1.84999100420957e-06\\
25	5.18011989834517e-06\\
26	2.16666858616632e-07\\
27	1.07611280069963e-05\\
28	1.48837651116202e-07\\
29	2.50930834129693e-06\\
30	3.10584829436657e-08\\
31	0.000254820104243881\\
32	1.59093588734045e-06\\
};
\addlegendentry{based on vector fitting}

\end{axis}
\end{tikzpicture}%
 \caption{Relative errors for $\mathcal{H}_2$ norm at optimal gain for Algorithm \ref{parameteroptimizationwithmethod1}  and Algorithm \ref{parameteroptimizationwithmethod2}} \label{c}  \end{center}
 \label{RelErrH2}
\end{figure}

Another important quantity to measure is the speed-up compared to the full problem.  Table 1 shows the average speed-ups for the optimization process obtained by both algorithms.

\begin{table}[h!]
\begin{center}
\begin{tabular}{|c||c|c|}
  \hline
      & Algorithm \ref{parameteroptimizationwithmethod1}& Algorithm \ref{parameteroptimizationwithmethod2} \\ \hline \hline
  Acceleration factor            &   7.8  & 60   \\\hline
  \hline
\end{tabular}
\end{center}
\caption{ Acceleration factors using surrogate optimization}
\end{table}
Both methods have optimized parameters with satisfactory relative errors with considerable acceleration of optimization process. For this damping optimization problem Algorithm \ref{parameteroptimizationwithmethod2}  not only produced more accurate results but also yielded bigger speed-up  than Algorithm \ref{parameteroptimizationwithmethod1}.  Therefore, for this problem,  Algorithm \ref{parameteroptimizationwithmethod2} was more efficient. 
However, we also note that 
Algorithm \ref{parameteroptimizationwithmethod1}  based on subsystem reduction include estimation of the true error $f(\pvec)$
as opposed to the sampling-based error 
$e(\pvec)$, which might improve the robustness of the optimization process.

\section{Conclusions}
We have introduced a framework for producing reduced order models of dynamical systems having an affine, low-rank parametric structure. The new framework does not require any sampling in the parameter domain and instead parametrically combines intermediate subsystems that are nonparametric. Our approach can guarantee uniform stability of the aggregated reduced model across the entire parameter domain in many cases.  Beyond the computational examples we provide for illustration, we show in some detail how this approach can be deployed efficiently in parameter optimization problems as well. 

\vspace{-2ex}
\section*{Acknowledgment}
\vspace{-1ex}
This work has been supported in parts by Croatian Science Foundation under the project `Vibration Reduction in Mechanical Systems' (IP-2019-04-6774) and project `Control of Dynamical Systems' (IP-2016-06-2468). 
The work of Beattie was supported in parts by NSF through Grant DMS-1819110. 
The work of Gugercin was supported in parts by NSF through Grants DMS-1720257 and  DMS-1923221. 

\vspace{-2ex}
\bibliographystyle{spmpsci}
\bibliography{sample}

\begin{thebibliography}{10}
\providecommand{\url}[1]{{#1}}
\providecommand{\urlprefix}{URL }
\expandafter\ifx\csname urlstyle\endcsname\relax
  \providecommand{\doi}[1]{DOI~\discretionary{}{}{}#1}\else
  \providecommand{\doi}{DOI~\discretionary{}{}{}\begingroup
  \urlstyle{rm}\Url}\fi

\bibitem{alla2019certified}
Alla, A., Hinze, M., Kolvenbach, P., Lass, O., Ulbrich, S.: A certified model
  reduction approach for robust parameter optimization with pde constraints.
\newblock Advances in Computational Mathematics pp. 1--30 (2019)

\bibitem{Antil2011}
Antil, H., Heinkenschloss, M., Hoppe, R.H.W.: Domain decomposition and balanced
  truncation model reduction for shape optimization of the {S}tokes system.
\newblock Optimization Methods and Software \textbf{26}(4--5), 643--669 (2011)

\bibitem{AntBG20}
Antoulas, A., , Beattie, C., Gugercin, S.: Interpolatory Methods for Model
  Reduction.
\newblock SIAM Publications, Philadelphia, PA (2020)

\bibitem{ANT05}
Antoulas, A.: Approximation of Large-Scale Dynamical Systems.
\newblock SIAM Publications, Philadelphia, PA (2005)

\bibitem{aca86}
Antoulas, A., Anderson, B.: On the scalar rational interpolation problem.
\newblock IMA J. of Mathematical Control and Information \textbf{3}, 61--88
  (1986)

\bibitem{Antoulas01asurvey}
Antoulas, A.C., Sorensen, D.C., Gugercin, S.: A survey of model reduction
  methods for large-scale systems.
\newblock Contemporary Mathematics \textbf{280}, 193--219 (2001)

\bibitem{Arian2002}
Arian, E., Fahl, M., Sachs, E.: Trust-region proper orthogonal decomposition
  models by optimization methods.
\newblock In: Proceedings of the 41st IEEE Conference on Decision and Control,
  pp. 3300--3305. Las Vegas, NV (2002)

\bibitem{baur2014mapping}
Baur, U., Beattie, C., Benner, P.: Mapping parameters across system boundaries:
  parameterized model reduction with low rank variability in dynamics.
\newblock PAMM \textbf{14}(1), 19--22 (2014)

\bibitem{morBauBBG11}
Baur, U., Beattie, C.A., Benner, P., Gugercin, S.: Interpolatory projection
  methods for parameterized model reduction.
\newblock SIAM J. Sci. Comput. \textbf{33}(5), 2489--2518 (2011)

\bibitem{BauB09}
Baur, U., Benner, P.: Model reduction for parametric systems using balanced
  truncation and interpolation.
\newblock at--Automatisierungstechnik \textbf{57}(8), 411--420 (2009)

\bibitem{BeattieGugercinSCL09}
Beattie, C., Gugercin, S.: Interpolatory projection methods for
  structure-preserving model reduction.
\newblock Systems and Control Letters \textbf{58}, 225--232 (2009)

\bibitem{beattie2009trm}
Beattie, C., Gugercin, S.: A trust region method for optimal {$\mathcal{H}_2$}
  model reduction.
\newblock In: Proceedings of 48th IEEE Conference on Decision and Control, pp.
  5370--5375 (2009)

\bibitem{BCOW17}
Benner, P., Cohen, A., Ohlberger, M., Willcox, K.: Model Reduction and
  Approximation: Theory and Algorithms.
\newblock Computational Science and Engineering, SIAM Publications,
  Philadelphia, PA (2017)

\bibitem{BennerGW15}
Benner, P., Gugercin, S., Willcox, K.: A survey of projection-based model
  reduction methods for parametric dynamical systems.
\newblock SIAM Review \textbf{57}(4), 483--531 (2015)

\bibitem{BennerKuerTomljTruh15}
Benner, P., K\"{u}rschner, P., Tomljanovi\'{c}, Z., Truhar, N.: Semi-active
  damping optimization of vibrational systems using the parametric dominant
  pole algorithm.
\newblock Journal of Applied Mathematics and Mechanics pp. 1--16 (2015).
\newblock {DOI}:10.1002/zamm201400158.

\bibitem{BenMS05}
Benner, P., Mehrmann, V., Sorensen, D.: Dimension Reduction of Large-Scale
  Systems.
\newblock Lecture Notes in Computational Science and Engineering,
  Springer-Verlag, Berlin/Heidelberg, Germany (2005)

\bibitem{BenSV14}
Benner, P., Sachs, E., Volkwein, S.: Model order reduction for {PDE}
  constrained optimization.
\newblock In: G.~Leugering, P.~Benner, S.~Engell, A.~Griewank, H.~Harbrecht,
  M.~Hinze, R.~Rannacher, S.~Ulbrich (eds.) Trends in {PDE} Constrained
  Optimization, \emph{International Series of Numerical Mathematics}, vol. 165,
  pp. 303--326. Springer International Publishing (2014).
\newblock \doi{10.1007/978-3-319-05083-6_19}.
\newblock \urlprefix\url{http://dx.doi.org/10.1007/978-3-319-05083-6_19}

\bibitem{BennerTomljTruh10}
Benner, P., Tomljanovi{\'c}, Z., Truhar, N.: Dimension reduction for damping
  optimization in linear vibrating systems.
\newblock Z. Angew. Math. Mech. \textbf{91}(3), 179 -- 191 (2011).
\newblock DOI: 10.1002/zamm.201000077

\bibitem{BennerTomljTruh11}
Benner, P., Tomljanovi{\'c}, Z., Truhar, N.: Optimal {D}amping of {S}elected
  {E}igenfrequencies {U}sing {D}imension {R}eduction.
\newblock Numer. Linear Algebr. \textbf{20}(1), 1--17 (2013).
\newblock DOI: 10.1002/nla.833

\bibitem{berljafa2017rkfit}
Berljafa, M., G\"{u}ttel, S.: The {RKFIT} algorithm for nonlinear rational
  approximation.
\newblock SIAM Journal on Scientific Computing \textbf{39}(5), 2049--2071
  (2017)

\bibitem{Blanchini12}
Blanchini, F., Casagrande, D., Gardonio, P., Miani, S.: Constant and switching
  gains in semi-active damping of vibrating structures.
\newblock Int. J. Control \textbf{85}(12), 1886--1897 (2012)

\bibitem{BuiThanh2008}
Bui-Thanh, T., Willcox, K., Ghattas, O.: Model reduction for large-scale
  systems with high-dimensional parametric input space.
\newblock SIAM Journal on Scientific Computing \textbf{30}(6), 3270--3288
  (2008)

\bibitem{Chinea-G.Talocia-2011}
Chinea, A., Grivet-Talocia, S.: On the parallelization of {Vector} {Fitting}
  algorithms.
\newblock IEEE Transactions on Components, Packaging and Manufacturing
  Technology \textbf{1}(11), 1761--1773 (2011)

\bibitem{desai1984transformation}
Desai, U., Pal, D.: A transformation approach to stochastic model reduction.
\newblock IEEE Transactions on Automatic Control \textbf{29}(12), 1097--1100
  (1984)

\bibitem{Drmac-Gugercin-Beattie:VF-2014-SISC}
Drma\v{c}, Z., Gugercin, S., Beattie, C.: Quadrature-based vector fitting for
  discretized $\mathcal{H}_2$ approximation.
\newblock SIAM J. Sci. Comp. \textbf{37}(2), A625--A652 (2015)

\bibitem{drmac2015vector}
Drma\v{c}, Z., Gugercin, S., Beattie, C.: Vector fitting for matrix-valued
  rational approximation.
\newblock SIAM Journal on Scientific Computing \textbf{37}(5), A2346--A2379
  (2015)

\bibitem{EggKLMM18}
Egger, H., Kugler, T., Liljegren-Sailer, B., Marheineke, N., Mehrmann, V.: On
  structure-preserving model reduction for damped wave propagation in transport
  networks.
\newblock SIAM J. Sci. Comput. \textbf{40}(1), A331--A365 (2018)

\bibitem{Feng05}
Feng, L.: Parameter independent model order reduction.
\newblock Math. Comput. Simulation \textbf{68}, 221–234 (2005)

\bibitem{morFenRK04}
Feng, L., Rudnyi, E.B., Korvink, J.G.: Parametric model reduction to generate
  boundary condition independent compact thermal model.
\newblock Technical report, IMTEK-Institute for Microsystem Technology (2004).
\newblock
  \urlprefix\url{http://modelreduction.com/doc/papers/feng04THERMINIC.pdf}

\bibitem{glover1984all}
Glover, K.: All optimal {Hankel}-norm approximations of linear multivariable
  systems and their $l_\infty$-error bounds.
\newblock International journal of control \textbf{39}(6), 1115--1193 (1984)

\bibitem{GVL89}
Golub, G.H., Loan, C.V.F.: Matrix Computations.
\newblock The Johns Hopkins University Press, Baltimore (1998)

\bibitem{gonnet2011robust}
Gonnet, P., Pach{\'o}n, R., Trefethen, L.N.: Robust rational interpolation and
  least-squares.
\newblock Electronic Transactions on Numerical Analysis \textbf{38}, 146--167
  (2011)

\bibitem{grimm2018parametric}
Grimm, A.R.: Parametric dynamical systems: Transient analysis and data driven
  modeling.
\newblock Ph.D. thesis, Virginia Tech (2018)

\bibitem{grivet2015passive}
Grivet-Talocia, S., Gustavsen, B.: Passive macromodeling: Theory and
  applications, vol. 239.
\newblock John Wiley \& Sons (2015)

\bibitem{gugercin2008isk}
Gugercin, S.: An iterative {SVD-Krylov} based method for model reduction of
  large-scale dynamical systems.
\newblock Linear Algebra and Its Applications \textbf{428}(8-9), 1964--1986
  (2008)

\bibitem{GugAB08}
Gugercin, S., Antoulas, A., Beattie, C.: $\mathcal{H}_2$ model reduction for
  large-scale linear dynamical systems.
\newblock SIAM Journal on Matrix Analysis and Applications \textbf{30}(2),
  609--638 (2008)

\bibitem{GugPBS12}
Gugercin, S., Polyuga, R., Beattie, C., Schaft, A.v.: Structure-preserving
  tangential interpolation for model reduction of port-{H}amiltonian systems.
\newblock Automatica \textbf{48}, 1963--1974 (2012)

\bibitem{Gustavsen-2006}
Gustavsen, B.: Improving the pole relocating properties of vector fitting.
\newblock IEEE Transactions on Power Delivery \textbf{21}(3), 1587--1592 (2006)

\bibitem{gustavsen1999rational}
Gustavsen, B., Semlyen, A.: Rational approximation of frequency domain
  responses by vector fitting.
\newblock IEEE Transactions on power delivery \textbf{14}(3), 1052--1061 (1999)

\bibitem{heinkenschloss2018reduced}
Heinkenschloss, M., Jando, D.: Reduced order modeling for time-dependent
  optimization problems with initial value controls.
\newblock SIAM Journal on Scientific Computing \textbf{40}(1), A22--A51 (2018)

\bibitem{hokanson2017projected}
Hokanson, J.M.: Projected nonlinear least squares for exponential fitting.
\newblock SIAM Journal on Scientific Computing \textbf{39}(6), A3107--A3128
  (2017)

\bibitem{morHunMS18}
Hund, M., Mlinari\'{c}, P., Saak, J.: An \(\mathcal{H}_2 \otimes
  \mathcal{L}_2\)-optimal model order reduction approach for parametric linear
  time-invariant systems.
\newblock Proc. Appl. Math. Mech. \textbf{18}(1), e201800084 (2018).
\newblock \doi{10.1002/pamm.201800084}

\bibitem{IonitaAntoulas14}
Ionita, A.C., Antoulas, A.C.: Data-driven parametrized model reduction in the
  {L}oewner framework.
\newblock SIAM J. Sci. Comput. \textbf{36}(3), 984--1007 (2014).
\newblock DOI: doi.org/10.1137/130914619

\bibitem{Kunisch2008}
Kunisch, K., Volkwein, S.: Proper orthogonal decomposition for optimality
  systems.
\newblock ESAIM: Mathematical Modelling and Numerical Analysis \textbf{42}(1),
  1--23 (2008)

\bibitem{KuzmTomljTruh12}
Kuzmanovi\'{c}, I., Tomljanovi\'{c}, Z., Truhar, N.: Optimization of material
  with modal damping.
\newblock Appl Math Comput \textbf{218}, 7326--7338 (2012)

\bibitem{mayo2007fsg}
Mayo, A., Antoulas, A.: {A framework for the solution of the generalized
  realization problem}.
\newblock Linear Algebra and Its Applications \textbf{425}(2-3), 634--662
  (2007)

\bibitem{moore1981principal}
Moore, B.: Principal component analysis in linear systems: Controllability,
  observability, and model reduction.
\newblock IEEE Transactions on Automatic Control \textbf{26}(1), 17--32 (1981)

\bibitem{MullerSchiehlen85}
M{\"u}ller, P., Schiehlen, W.: Linear Vibrations.
\newblock Martinus Nijhoff Publishers (1985)

\bibitem{mullis1976synthesis}
Mullis, C., Roberts, R.: Synthesis of minimum roundoff noise fixed point
  digital filters.
\newblock IEEE Transactions on Circuits and Systems \textbf{23}(9), 551--562
  (1976)

\bibitem{nakatsukasa2018aaa}
Nakatsukasa, Y., S{\`e}te, O., Trefethen, L.N.: The {AAA} algorithm for
  rational approximation.
\newblock SIAM Journal on Scientific Computing \textbf{40}(3), A1494--A1522
  (2018)

\bibitem{NTT19}
Nakić, I., Tomljanović, Z., Truhar, N.: Mixed control of vibrational systems.
\newblock Journal of Applied Mathematics and Mechanics \textbf{99}(9), 1--15
  (2019)

\bibitem{ober1991balanced}
Ober, R.: Balanced parametrization of classes of linear systems.
\newblock SIAM Journal on Control and Optimization \textbf{29}(6), 1251--1287
  (1991)

\bibitem{morwiki_thermal}
{Oberwolfach Benchmark Collection}: Thermal model.
\newblock hosted at {MORwiki} -- Model Order Reduction Wiki (20XX).
\newblock \urlprefix\url{http://modelreduction.org/index.php/Thermal_Model}

\bibitem{opdenacker1988contraction}
Opdenacker, P., Jonckheere, E.: A contraction mapping preserving balanced
  reduction scheme and its infinity norm error bounds.
\newblock IEEE Transactions on Circuits and Systems \textbf{35}(2), 184--189
  (1988)

\bibitem{van2019parametric}
van Ophem, S., Deckers, E., Desmet, W.: Parametric model order reduction
  without a priori sampling for low rank changes in vibro-acoustic systems.
\newblock Mechanical Systems and Signal Processing \textbf{130}, 597--609
  (2019)

\bibitem{Penzl99}
Penzl, T.: Numerische simulation auf massiv parallelen rechnern.
\newblock Ph.D. thesis, TU Chemnitz (1999).
\newblock Algorithms for Model Reduction of Large Dynamical Systems, Tech. Rep.
  SFB393/99-40

\bibitem{polyuga2011structure}
Polyuga, R., van~der Schaft, A.: Structure preserving moment matching for
  port-hamiltonian systems: {A}rnoldi and {L}anczos.
\newblock IEEE Transactions on Automatic Control \textbf{56}(6), 1458--1462
  (2011)

\bibitem{QuarteroniMN16}
Quarteroni, A., Manzoni, A., Negri, F.: Reduced Basis Methods for Partial
  Differential Equations: An Introduction.
\newblock R. {UNITEXT}. Springer Cham (2016)

\bibitem{RommesM08}
Rommes, J., Martins, N.: {Computing Transfer Function Dominant Poles of
  Large-Scale Second-Order Dynamical Systems}.
\newblock SIAM J. Sci. Comput. \textbf{30}(4), 2137--2157 (2008)

\bibitem{Sanathanan-Koerner-1963}
Sanathanan, C., Koerner, J.: Transfer function synthesis as a ratio of two
  complex polynomials.
\newblock IEEE Trans. Autom. Control \textbf{8}(1), 56--58 (1963)

\bibitem{TomljBeattieGugercin18}
Tomljanovi\'{c}, Z., Beattie, C., Gugercin, S.: Damping optimization of
  parameter dependent mechanical systems by rational interpolation.
\newblock Advances in Computational Mathematics pp. 1--24 (2018)

\bibitem{VES2011}
Veseli\'{c}, K.: Damped {O}scillations of {L}inear {S}ystems.
\newblock Springer {L}ecture {N}otes in {M}athematics, Springer-Verlag, Berlin
  (2011)

\bibitem{YueMeer13SIAM}
Yue, Y., Meerbergen, K.: Accelerating optimization of parametric linear systems
  by model order reduction.
\newblock SIAM Journal on Optimization \textbf{23}(2), 1344--1370 (2012)

\bibitem{yue2013}
Yue, Y., Meerbergen, K.: Accelerating optimization of parametric linear systems
  by model order reduction.
\newblock SIAM Journal on Optimization \textbf{23}(2), 1344--1370 (2013)

\end{thebibliography}

\end{document}